\documentclass[bj,preprint]{imsart}
\usepackage{amsmath,amsthm,amsfonts,graphicx}
\usepackage[psamsfonts]{amssymb}
\usepackage[OT1]{fontenc}
\usepackage{mathtools}
\usepackage{bbm}
\usepackage{hyperref}
\usepackage{epstopdf}
\newtheorem{thm}{Theorem}[section]
\newtheorem{lem}{Lemma}[section]

\newtheorem{prop}{Proposition}[section]
\newtheorem{ass}{Assumption}
\newtheorem{cor}{Corollary}[section]
\renewcommand{\c}{\mathbf{c}}
\newcommand{\m}{\mathbf{m}}
\newcommand{\cl}{\hat{\c}_{n,\lambda}}
\def\Var{\mathop{\rm Var}}
\newcommand{\argmin}{\operatornamewithlimits{argmin}}

\begin{document}
\begin{frontmatter}
\title{Sparse oracle inequalities for variable selection via regularized quantization}
\runtitle{Sparse oracle inequalities for variable selection via regularized quantization}

\begin{aug}
\author{\fnms{Cl\'{e}ment} \snm{ Levrard} \ead[label=e1]{levrard@math.univ-paris-diderot.fr}}

\runauthor{C. Levrard}
\affiliation{LPMA-Universit\'e Paris Diderot}
\address{Universit\'e Paris Diderot, 8 place Aur\'elie Nemours, 75013 Paris \printead{e1}}
\end{aug}

\begin{abstract}
      We give oracle inequalities on  procedures which combines quantization and variable selection via a weighted Lasso $k$-means type algorithm. The results are derived for a general family of weights, which can be tuned to size the influence of the variables in different ways. Moreover, these theoretical guarantees are proved to adapt the corresponding sparsity of the optimal codebooks, suggesting that these procedures might be of particular interest in high dimensional settings. Even if there is no sparsity assumption on the optimal codebooks, our procedure is proved to be close to a sparse approximation of the optimal codebooks, as has been done for the Generalized Linear Models in regression. If the optimal codebooks have a sparse support, we also show that this support can be asymptotically recovered, providing an asymptotic consistency rate. These results are illustrated with Gaussian mixture models in arbitrary dimension with sparsity assumptions on the means, which are standard distributions in model-based clustering. 
      
      
\end{abstract}

\end{frontmatter}

\section{Introduction}

       Let $P$ be a distribution over $\mathbb{R}^d$. Quantization is the problem of replacing $P$ with a finite set of points, without loosing too much information. To be more precise, if $k$ denotes an integer, a $k$ points quantizer $Q$ is defined as a map from $\mathbb{R}^d$ into a finite subset of $\mathbb{R}^d$ with cardinality $k$. In other words, a $k$-quantizer divide $\mathbb{R}^d$ into $k$ groups, and assigns each group a representative, providing both a compression and a classification scheme for the distribution $P$.

        The quantization theory was originally developed as a way to answer signal compression issues in the late 40's (see, e.g., \cite{Gersho91}). However, unsupervised classification is also in the scope of its application. Isolating meaningful groups from a cloud of data is a topic of interest in many fields, from social science to biology.

       Assume that $P$ has a finite second moment, and let $Q$ be a $k$ points quantizer. The performance of $Q$ in representing $P$ is measured by the distortion
       \[
       R(Q) = P \| x - Q(x) \|^2,
       \]
       where $Pf$ means integration of $f$ with respect to $P$. It is worth pointing out that many other distortion functions can be defined, using $\| x - Q(x) \|^r$ or more general distance functions (see, e.g., \cite{Fischer10} or \cite{GL00}). However, the choice of the Euclidean squared norm is convenient, since it allows to fully take advantage of the Euclidean structure of $\mathbb{R}^d$, as described in \cite{Levrard14}. Moreover, from a practical point of view, the $k$-means algorithm (see \cite{Lloyd82}) is designed to minimize this squared-norm distortion and can be easily implemented.
       
       Since the distortion is based on the Euclidean distance between a point and its image, it is well known that only nearest-neighbor quantizers are to be considered (see, e.g., \cite{GL00} or \cite{Pollard82}). These quantizers are quantizers of the type $x \mapsto \argmin_{j=1, \hdots, k}{\| x - c_j \|}$, where the $c_i$'s are elements of $\mathbb{R}^d$ and are called code points. A vector of code points $(c_1, \hdots, c_k)$ is called a codebook, so that the distortion takes the form
       \[
       R(\c) = P \min_{j=1, \hdots, k} {\| x - c_j \|^2}.
       \] 
      It has been proved in \cite{Pollard81} that, whenever $P\|x\|^2 < \infty$, there exists optimal codebooks, denoted by $\c^*$.
      
      Let $X_1, \hdots, X_n$ denote an independent and identically distributed sample drawn from $P$, and denote by $P_n$ the associated empirical distribution, namely, for every measurable subset $A$, $P_n(A) = 1/n \left | \{ i | X_i \in A \} \right |$. The aim is to design a codebook from this $n$-sample, whose distortion is as close as possible to the optimum $R(\c^*)$. The $k$-means algorithm provides the empirical codebook $\hat{\c}_n$, defined by
      \[      
      \hat{\c}_n = \argmin{\frac{1}{n}\sum_{i=1}^{n} \min_{j=1, \hdots, k}{\| X_i - c_j \|^2}} = \argmin {P_n \min_{j=1, \hdots, k} {\| x - c_j \|^2}}.
      \]
      Unfortunately, if $P^{(p)} \neq 0$, where $P^{(p)}$ denotes the marginal distribution of $P$ on the $p$-th coordinate, then $\hat{\c}_n^{(p)} = (\hat{c}_1^{(p)}, \hdots, \hat{c}_k^{(p)})$ may not be zero, even if the $p$-th coordinate has no influence on the classification provided by the $k$-means. For instance, if $\c^{*,(p)}=0$, and $P^{(p)}$ has a density, then $\hat{\c}_n^{(p)} \neq 0$ almost surely. This suggests that the $k$-means algorithm does not provide sparse codebooks, even in the case where some variables plays no role in the classification, which can be detrimental to the computational tractability and to the interpretation of the corresponding clustering scheme in high-dimensional settings.
      
      Consequently, when $d$ is large, a variable selection procedure is usually performed preliminary to the $k$-means algorithm. The variable selection can be achieved using penalized BCCS strategies, as exposed in \cite{Chang14} or \cite{Witten10}. Though these procedures offer good performance in classifying the sample $X_1, \hdots, X_n$, under the assumption that the marginal distributions $P^{(p)}$ are independent, no theoretical result on the prediction performance has been given. An other way to perform variable selection can be to select coordinates whose empirical variances are larger than a determined ratio of the global variance, following the idea of \cite{Steinley08}. This algorithm has shown good results on practical examples, such as curve clustering (see, e.g., \cite{Antoniadis10}). However, there is no theoretical result on the prediction performance of the selected coordinates. 
      
      Algorithms combining variable selection through PCA and clustering via $k$-means, like RKM (Reduced $k$-means, introduced in \cite{DeSoete94}) and FKM (Factorial $k$-means, introduced in \cite{Vichi01}), are also very popular in practice. Some results on the performance in classifying the sample $X_1, \hdots, X_n$ have been derived in \cite{Timmerman10} under strong conditions on $P$. In addition, some asymptotic prediction results on these procedures have been established in \cite{Terada12} and \cite{Terada13}, showing that both the resulting codebook and its  distortion converge almost surely to respectively a minimizer of the distortion constrained on a lower-dimensional subspace of $\mathbb{R}^d$ and the distortion of the latter, following the approach of \cite{Pollard81}. However, these methods could be unsuitable for interpreting which variables are relevant for the clustering.  In addition, no bounds on the excess distortion are available to our knowledge, and the choice of the dimension of the reduction space remains a hard issue, tackled in our procedure by a $L_1$-type penalization.
      
      In fact, excess risk bounds for procedures combining dimensionality reduction and clustering are mostly to be found in the model-based clustering literature (see, e.g., \cite{Michel13} for a $L_0$-type penalization method, and \cite{Meynet13} for a $L_1$-type penalization method). This approach, consisting in modeling $P$ via a Gaussian mixture with sparse means through density estimation via constrained Maximum Likelihood Estimators, is clearly connected to ours. In fact, most of the derivation for the oracle inequalities stated in this paper use the same tools, drawn from empirical process theory. Nevertheless, no results on the convergence of the estimated means (i.e., model consistency) have been derived in this framework, and this model-based approach theoretically fails when $P$ is not continuous, unlike $k$-means one (see, e.g., \cite{Levrard14}).
      
      This paper exposes a theoretical study of a weighted Lasso type procedure adapted to $k$-means, as suggested in \cite{Sun12}. Results are given for a general family of weights, encompassing the weights proposed in \cite{Sun12} as well as those proposed in \cite{vandeGeer08} in a Generalized Linear Models for regression setting. To be more precise, we provide non-asymptotic excess distortion bounds along with model consistency results, under weaker conditions than ones required in \cite{Sun12} (for instance, the coordinates are not assumed to be independant), and adapting the sparsity of the optimal codebooks. From these non-asymptotic bounds, some asymptotic rates of convergence are derived when both the dimension and the sample size are large, showing that these Lasso type procedures may be suitable for high dimensional quantization. Interestingly, the excess distortion bounds are valid in the case where it may exist several optimal codebooks, contrary to results in \cite{Sun12} and \cite{vandeGeer08}. These results are illustrated with Gaussian mixture distributions, often encountered in model-based clustering literature, showing at the same time that optimal codebooks can be proved to be unique for this type of distributions, under some conditions on the variances of the components of the mixture.

      The paper is organized as follows. Some notation are introduced in Section \ref{Notation}, along with the Lasso $k$-means procedure and the different assumptions. The consistency and prediction results are gathered in Section \ref{Results}, and the proof of these results are exposed in Section \ref{Proofs}. At last, the proofs of some auxiliary results are  given in the \nameref{Appendix} section.
  
\section{Notation}\label{Notation}

     Let $x$ be in $\mathbb{R}^d$, then the $p$-th coordinate of $x$ will be denoted by $x^{(p)}$. Throughout this paper, it is assumed that, for every $p = 1, \hdots, d$, there exists a sequence $M_p$, such that $| x^{(p)} | \leq M_p$ $P$-almost surely. In other words $P$ is assumed to have bounded marginal distributions $P^{(p)}$. To shorten notation, the Euclidean coordinate-wise product $\prod_{p=1}^{d}{ [-M_p, M_p ]}$ will be denoted by $C$. To frame quantization as a contrast minimization issue, let us introduce the following contrast function
     
     \begin{align*}
          \gamma : \left \{
                 \begin{array}{@{}ccl@{}}
                 \left (\mathbb{R}^d \right )^k \times \mathbb{R}^d& \longrightarrow & \mathbb{R} \\
                 \hspace{0.45cm}  (\mathbf{c},x) & \longmapsto & \underset{j=1, \hdots, k}{\min}{\left \| x-c_j \right \|^2}
                 \end{array}
                 \right . ,
          \end{align*}
          where $\c = (c_1, \hdots, c_k)$ denotes a codebook, that is a $k d$-dimensional vector. The risk $R(\c)$ then takes the form $R(\c)= R(Q) = P{\gamma}(\c,.)$, where we recall that $Pf$ denotes the integration of the function $f$ with respect to $P$. Similarly, the empirical risk $\hat{R}_n(\c)$ can be defined as $\hat{R}_n(\c) = P_n{\gamma}(\c,.)$, where $P_n$ is the empirical distribution associated with $X_1, \hdots, X_n$, in other words $P_n(A) = 1/n \left | \{i |X_i \in A \} \right |$, for every measurable subset $A \subset \mathbb{R}^d$. The usual $k$-means codebook $\hat{\c}_n$ is then defined as a minimizer of $\hat{R}_n(\c)$.

					It is worth pointing out that, since the support of $P$ is bounded, then there exist such minimizers $\hat{\c}_n$ and $\mathbf{c}^*$ (see, e.g., Corollary 3.1 in \cite{Fischer10}). In the sequel, the set of minimizers of the risk $R(.)$ will be denoted by $\mathcal{M}$. Then, for any codebook $\c$, the loss $\ell(\c,\c^*)$ may be defined as the excess distortion, namely $\ell(\c,\c^*) = R(\c) - R(\c^*)$, for $\c^*$ in $\mathcal{M}$.


            From now on we assume that $k\geq 2$. Let $c_1, \hdots, c_k$ be a sequence of code points. A central role is played by the set of points which are closer to $c_i$ than to any other $c_j$'s. To be more precise, the Voronoi cell, or quantization cell associated with $c_i$ is the closed set defined by
         \begin{align*}\label{Voronoidefinition}
         V_i(\c)= \left \{ x \in \mathbb{R}^d | \quad  \forall j \neq i \quad  \|x-c_i\| \leq \|x-c_j\|\right \}. 
         \end{align*}
         It may be noted that $(V_1(\c), \hdots, V_k(\c))$ does not form a partition of $\mathbb{R}^d$, since $V_i(\c) \cap V_j(\c) $ may be non empty. To address this issue, the Voronoi partition associated with $\c$ is defined as the sequence of subsets $W_i(\c)= V_i(\c) \setminus (\cup_{i >j} V_j(\c))$, for $i=1, \hdots, k$. It is immediate that the $W_i(\c)$'s form a partition of $\mathbb{R}^d$, and that for every $i=1, \hdots, k$,
         \[
         \bar{W}_i(\c)=V_i(\c),
         \]
         where $\bar{W}_i(\c)$ denotes the closure of the subset $W_i(\c)$. The open Voronoi cell is defined the same way by
         \begin{align*}
         \overset{o}{V}_i(\c)=\left \{ x \in \mathbb{R}^d | \quad  \forall j \neq i \quad  \|x-c_i\| < \|x-c_j\|\right \},
         \end{align*}          
          and the following inclusion holds, for $i$ in $\left \{1,\hdots,k\right \}$,
          \[
          \overset{o}{V}_i(\c) \subset W_i(\c) \subset V_i(\c).
          \]
          The risk $R(\c)$ then takes the form
         \[
         R(\c)= \sum_{i=1}^{k}{P\left ( \|x-c_i\|^2 \mathbbm{1}_{W_i(\c)}(x) \right )},
         \]
         where $\mathbbm{1}_A$ denotes the indicator function associated with $A$. In the case where $P(W_i(\c)) \neq 0$, for every $i=1, \hdots, k$, it is clear that
         \[
         P(\|x - c_i\|^2 \mathbbm{1}_{W_i(\c)}(x)) \geq P(\|x - \eta_i\|^2 \mathbbm{1}_{W_i(\c)}(x)),
         \]
         with equality only if $c_i = \eta_i$, where $\eta_i$ denotes the conditional expectation of $P$ over the subset $W_i(\c)$, that is 
         \[
         \eta_i = \frac{P(x\mathbbm{1}_{W_i(\c)}(x))}{P(W_i(\c))}.
         \] 
         Moreover, it is proved in Proposition 1 of \cite{Graf07} that, for every Voronoi partition $W(\c^*)$ associated with an optimal codebook $\c^*$, and every $i=1, \hdots, k$, $P(W_i(\c^*)) \neq 0$. Consequently, any optimal codebook satisfies the so-called centroid condition (see, e.g., Section 6.2 of \cite{Gersho91}), that is
         \[
         \c^*_i= \frac{P(x\mathbbm{1}_{W_i(\c^*)}(x))}{P(W_i(\c^*))}.
         \]
         As a remark, the centroid condition ensures that $\mathcal{M} \subset C^k$, and, for every $\c^*$ in $\mathcal{M}$, $i \neq j$, 
         \begin{align*}
         P(V_i(\c^*) \cap V_j(\c^*)) 
				& = P\left ( \left \{ x \in \mathbb{R}^d| \quad \forall i' \quad \|x-c_i^*\|= \|x - c_j^*\| \leq \| x - c^*_{i'}\| \right \} \right ) \\
				&= 0.
         \end{align*}
         A proof of this statement can be found in Proposition 1 of \cite{Graf07}. 
         According to \cite{Levrard14}, for every $\c^*$ in $\mathcal{M}$, the following set is of special interest:
         \begin{align*}
         N_{\c^*} = \bigcup_{i \neq j}{V_i(\c^*) \cap V_j(\c^*)}.
         \end{align*}
         To be more precise, the key quantity is the margin function, which is defined as
          \[
         p(t) = \sup_{\c^* \in \mathcal{M}}P(N_{\c^*}(t)),
         \] 
         where $N_{\c^*}(t)$ denotes the $t$-neighborhood of $N_{\c^*}$. As shown in \cite{Levrard14}, bounds on this margin function (see Assumption \ref{margincondition} below) can provide interesting results on the convergence rate of the $k$-means codebook, along with basic properties of optimal codebooks.

					In order to perform both variable selection and quantization, we introduce the Lasso $k$-means codebook $\cl$ as follows.
				\begin{align}\label{deflassokmeans}
				\cl \in \argmin_{\c \in C^k}{P_n \gamma(\c,.) + \lambda I_{\hat{w}}(\c)},
				\end{align}
				where $\hat{w}$ is a possibly random sequence of weights of size $d$, and $I_{\hat{w}}()$ denotes the penalty function
				\begin{align}\label{defpenalty}
			   	I_{\hat{w}}(\c) = \sum_{p=1}^{d}{\hat{w}_p  \| \c^{(p)}  \|}.
\end{align}		
Let us recall here that $\c^{(p)}$ denote the vector $(c_1^{(p)}, \hdots, c_k^{(p)})$ made of the $p$-th coordinates of the different codepoints. The results exposed in the following section are illustrated with three sequences of weights, corresponding to different codebooks: the \emph{plain Lasso} codebook, defined by the deterministic sequence $\hat{w}_p = 1$, the \emph{normalized Lasso} codebook, defined by $\hat{w}_p = \hat{\sigma}_p$, and the \emph{threshold Lasso} codebook, which is a slight modification of the original Lasso-type procedure mentioned in \cite{Sun12} and is defined by $\hat{w}_p = 1/(\delta \vee \| \hat{\c}_n^{(p)} \|)$, where $\hat{\c}_n$ denotes the $k$-means codebook and $\delta$ a parameter to be tuned. It is likely that other families of weights may be of special interest, for instance combining normalization and threshold. Consequently the results are derived for an arbitrary family of weights satisfying some convergence conditions.

				These $L_1$-type penalties have been designed to drive the irrelevant $(p)$-th coordinates $c_1^{(p)}, \hdots, c_k^{(p)}$ together to zero (see, e.g., \cite{Bach08}), according to different criterions. Note that this kind of penalties is well-adapted to centered distributions. In practice, centering the data provides codebooks of the form $(\hat{c}_{n,\lambda,1} + \bar{X}, \hdots, \hat{c}_{n,\lambda,k} + \bar{X})$ for the non centered distribution, where $\bar{X}$ denotes the empirical mean and $\cl$ is hopefully sparse. From a theoretical point of view, deriving how close the codebook $\cl$ computed on the centered data is to a codebook $\c^*-m$ would require a bound on $\|\bar{X}-m\|$, where $m$ is the mean of $P$. In our framework, such a bound is typically of order $\sqrt{d/n}$ (see, e.g., Figure \ref{Lkmeansdistortions}), hence might be unsuited for high dimensional settings. However, in some particular cases (for instance when the mean is sparse), other estimators of the means that are adapted to the high dimensional framework could be combined with our procedure.  
				
				  To describe the influence of the different coordinates, the following notation are adopted. Let $S \subset \{ 1, \hdots, d \}$ denote a subset of coordinates, then for any vector $x$ in $(\mathbb{R}^d)^\ell$ and set $A$ $\subset$ $(\mathbb{R}^d)^\ell$, $\ell$ being a positive integer, $x_S$ will denote the vector in $(\mathbb{R}^{|S|})^{\ell}$ corresponding to the coefficients of $x$ on variables in $S$, and $A_S$ will denote the set of such $x_S$, for $x$ in $A$. Moreover, let $P^{S}$ denote the marginal distribution of $P$ over the set $\mathbb{R}^{|S|}$. We may then define the restricted distortions and variances as follows:
					\begin{align*}
					\left \{
					\begin{array}{@{}ccc}
					\sigma_{S}^2 &=& P^{S} \|x\|^2, \\
					\hat{\sigma}_{S}^2 &=& P_n^{S} \|x\|^2, \\
					R_{S}^* &=& \min_{\c \in C_{S}}{P^{S} \gamma(\c,.)}, \\
					\hat{R}_{S} &=& \min_{\c \in C_{S}}{P_n^{S} \gamma(\c,.)}, 
					\end{array}
					\right.
					\end{align*}
					where the vector $x$ is element of $\mathbb{R}^{|S|}$. Elementary properties of the distortion show that, if $S = S_1 \cup S_2$, with empty intersection, then
					\begin{align}\label{sousadditivité}
					\left \{
					\begin{array}{@{}ccc}
					\sigma_{S}^2 &=& \sigma^2_{S_1} + \sigma^2_{S_2}, \\
					\hat{\sigma}_{S}^2 &=& \hat{\sigma}_{S_1}^2 + \hat{\sigma}_{S_2}^2, \\
					R_{S}^* &\geq& R_{S_1}^* + R_{S_2}^*, \\
					\hat{R}_{S} &\geq& \hat{R}_{S_1} + \hat{R}_{S_2}. 
					\end{array}
					\right.
					\end{align}
					These elementary properties will be of importance when choosing which coordinate to select. A special attention will be paid to the subsets of variables formed by the support of codebooks. To be more precise, for every codebook $\c$ in $C^k$, we define the support $S(\c)$ of $\c$ by $S(\c) = \{ j \in \left \{1, \hdots, d \right \} | \c^{(j)} \neq 0  \}$.  The following Proposition gives a first glance at which variables are in $S(\cl)$.
				
				\begin{prop}\label{KKTlasso}
				Let $p$ be in $\{ 1, \hdots, d\}$. If
				\[
				\sqrt{\hat{\sigma}_p^2 - \hat{R}_p} < \frac{\hat{w}_p \lambda}{2},
				\]
				then
				\[
				\cl^{(p)} = (\hat{c}_{n,\lambda,1}^{(p)}, \hdots, \hat{c}_{n,\lambda,k}^{(p)}) = (0, \hdots, 0).
				\]
				\end{prop} 
									According to Proposition \ref{KKTlasso}, the Lasso $k$-means procedures may be thought of as a multimodularity test on every coordinate, in the spirit of \cite{Jin14}. This result ensures that, if the distortion of the codebook $(0, \hdots, 0)$ is close to the optimal empirical distortion, on the $p$-th coordinate, then the Lasso $k$-means will drive the $p$-th variable to $0$. For the plain Lasso, the differences $\sqrt{\hat{\sigma}_p^2 - \hat{R}_p}$ are uniformly thresholded, whereas for the normalized Lasso, the threshold in $\lambda$ is applied on the ratios $ \hat{R}_p/ \hat{\sigma}_p^2$. This point suggests that the  normalized Lasso may succeed in recovering informative variables with small ranges. 	
					
				We introduce now the assumptions which will be required to derive theoretical results on the performance of the Lasso codebooks. To deal with the case of possibly several optimal codebooks, we introduce the following structural assumption on $P$.
				
				\begin{ass}\label{nosubcodebook}
				For every $\c^*$ in $\mathcal{M}$ and $\c$ in $C^k$, if $S(\c) \subsetneq S(\c^*)$, then $R(\c) > R(\c^*)$.
				\end{ass} 
				
				Assumption \ref{nosubcodebook} roughly requires that no optimal codebook has a support strictly contained in the support of another optimal codebook. This is obviously the case if $P$ has a unique optimal codebook, up to relabeling. 	 
									
				\begin{ass}[Margin Condition]\label{margincondition}
				There exists $r_0 >0$ such that
				\begin{align}\label{marginweight}
				\forall t \leq r_0 \quad p(t) \leq c_0(P) t,
				\end{align}
				where $c_0(P)$ is a fixed constant, defined in \cite{Levrard14}.
				\end{ass}
				
				As exposed in \cite{Levrard14}, Assumption \ref{margincondition} may be thought of as a margin condition for squared distance based quantization. Some examples of distributions satisfying \eqref{marginweight} are given in \cite{Levrard14}, including Gaussian mixtures under some conditions. Roughly, if $P$ is well concentrated around $k$ poles, then \eqref{marginweight} will hold. It is also worth mentioning that the condition required in \cite{Sun12} seems stronger than the condition required in Assumption \ref{margincondition}, since it requires $P$ to have a unique optimal codebook, to be a mixture of components centered on the different optimal code points, and that the Hessian matrix of the risk function located at the optimal codebook is positive definite. 
				
				Moreover, Assumption \ref{margincondition} is a sufficient condition to ensure that some elementary properties that are often assumed are satisfied, as described in the following Proposition.
				\begin{prop}\label{TMC}
				If $P$ satisfies Assumption \ref{margincondition}, then
				\begin{itemize}
				\item[$i)$] $\mathcal{M}$ is finite,
				\item[$ii)$] there exists $\kappa'_0 >0$ such that, for every $\c$ in $C^k$, $\|\c - \c^*(\c)\|^2 \leq \kappa'_0 \ell(\c,\c^*)$,
				\end{itemize}
				where $\c^*(\c) \in \argmin_{\c^*}{\| \c - \c^*\|}$.
				
				 Moreover, if $P$ satisfies Assumption \ref{nosubcodebook}, then there exists a constant $\kappa''_0$ such that, for every $\c^*$ in $\mathcal{M}$ and $S(\c) \subsetneq S(\c^*)$, we have
				\begin{align*}
				\| \c - \c^* \|^2 \leq \kappa_0'' \ell(\c,\c^*).
				\end{align*}
				\end{prop}
				The two first statements of Proposition \ref{TMC} are to be found in Proposition 2.2 of \cite{Levrard14}, the proof of the third statement is given in Section \ref{ProofofPropositionTMC}. Proposition \ref{TMC} may be thought of as a generalization of the positive Hessian matrix condition of \cite{Pollard82} to the non-continuous case. It also allows to deal with the case where $P$ has several optimal codebooks. In the following, we denote by $\kappa_0$ the quantity $\kappa'_0 \vee \kappa''_0$, whenever Assumption \ref{margincondition} and Assumption \ref{nosubcodebook} are satisfied.				
				
				In addition to Assumption \ref{margincondition}, we assume that the weights $\hat{w}_p$ satisfy a uniform concentration inequality around some deterministic weights, as stated below. 
				\begin{ass}[Weights concentration]\label{WeightsAssumption}
				There exist deterministic weights $w_p >0$, $p=1, \hdots, d$, and a constant $0 \leq \kappa_1 < 1$ such that
				\begin{align}\label{Weightsinequality}
				\mathbb{P} \left ( \sup_{p = 1, \hdots, d} \left | \frac{\hat{w}_p}{w_p} - 1 \right | > \kappa_1 \right ) := r_1(n) \underset{n \rightarrow \infty} \longrightarrow 0.
				\end{align}
				\end{ass}
				 
				  Assumption \ref{WeightsAssumption} is obviously satisfied for the plain Lasso ($\hat{w}_p=1$). The following proposition ensures that this statement remains true for the two other examples of weights. For any sequence $w_p$, we denote by $T(w)$ the quantity $\sup_{p=1, \hdots,d}{M_p/w_p}$. With a slight abuse of notation, $T(\sigma)$ and $T(\delta)$ will refer to the sequences $\sigma_p$ and $1/(\| \c^{*,(p)} \| \vee \delta)$, where the latter is well defined when $P$ has a unique optimal codebook.
				
				\begin{prop}\label{trueweights} \ 
				
				For $\hat{w}_p = \hat{\sigma}_p$, if $1>\kappa_1 > \frac{T^2(\sigma) \sqrt{\log(d)}}{\sqrt{2n}}$, then Assumption \ref{WeightsAssumption} holds with $w_p = \sigma_p$ and $r_1(n) = e^{-\left( \frac{\sqrt{2n} \kappa_1}{T^2(\sigma)} - \sqrt{\log(d)} \right )^2}$.

				For $\hat{w}_p = 1/(\|\hat{\c}_n^{(p)} \| \vee \delta)$, let $M$ be defined as $M = \sqrt{M_1^2 + \hdots + M_d^2}$. If $1 > \kappa_1 > C_0 \frac{M \sqrt{k}}{\sqrt{n}\delta}$, for a fixed constant $C_0$, Assumption \ref{margincondition} is satisfied, and $\c^*$ is unique (up to relabeling), then Assumption \ref{WeightsAssumption} holds with $w_p = 1/(\| \c^{*,(p)}\| \vee \delta)$ and $r_1(n)=e^{- \left ( \frac{n \delta^2 \kappa_1^2 }{C_0^2 M^2}- k \right)}$.

				\end{prop}
				
				The proof of Proposition \ref{trueweights} follows from standard concentration inequalities, and can be found in the Section \ref{proofofpropositiontrueweights} of the Appendix. At first sight, the assumption that $\c^*$ is unique seems quite restrictive. However, as exposed in Section \ref{Gaussianmixture}, it can be shown that Gaussian mixtures satisfy this property, provided that the variances of the components are small enough. In fact, if $P$ has several optimal codebooks, there is no intuition about toward which one $\hat{\c}_n$ will converge, hence the difficulty of defining deterministic limit weights for $\hat{w}_p$. 
				
				At last, we define the following quantities $\lambda_0$ and $\lambda_1$ which will play the role of minimal values for the regularization parameter $\lambda$, as exposed in \cite{vandeGeer08}.
				\begin{align}\label{minimalregularization}
				\left \{
				\begin{array}{@{}ccc}
				\lambda_0&=& 8 \sqrt{2 \pi} \sqrt{\frac{k \log(kd)}{n}} T(w), \\
				\lambda_1(x)&=& e \lambda_0 \left ( 1 + \sqrt{\frac{u+x}{k \log(kd)}} \right),
				\end{array}
				\right .
				\end{align}
		where $x>0$ and $u = \log \left ( \frac{\| w \|_2^2 \sqrt{n}}{\sqrt{\log(kd)}} \right )$. These two quantities come from empirical process theory, their roles are explained in Section \ref{Proofs}. Roughly, $\lambda_0$ is the minimal value of the regularization parameter which ensures that the empirical risk is close to the true risk uniformly on $C^k$, and $\lambda_1(x)$ is the minimal value which ensures that the deviation between empirical and true risk may be compared to the norm $I_w$ uniformly on $C^k$.

\section{Results}\label{Results}
       We recall here that $k \geq 2$. The case $k=1$ may be treated as a special case of the standard Lasso estimator for linear regression (see, e.g., Chapter 2 of \cite{VandeGeer11}).
\subsection{\texorpdfstring{Sparsity adaptive slow rate of convergence for the distortion}{Sparsity adaptive slow rate of convergence for the distortion}}\label{Slowrates}

       Following the approach of \cite{Meynet11}, Lasso type procedures may be thought of as model selection procedures over $L_1$ balls. Theorem \ref{SML} below is the adaptation of this idea for the Lasso $k$-means procedures.
       
       \begin{thm}\label{SML}
       Suppose that Assumption \ref{WeightsAssumption} is satisfied, for some constant $\kappa_1<1$, and choose
       \[
       \lambda \geq \frac{\lambda_1(x)}{1-\kappa_1},
       \]
       for some $x>0$, where $\lambda_1$ is defined in \eqref{minimalregularization}. 
       Then, with probability larger than $1- r_1(n)-e^{-x}$, for every $\c^* $ in $\mathcal{M}$, we have
       \[
       \ell(\cl,\c^*) \leq \inf_{r>0} \inf_{I_w(\c) \leq r} \left ( \ell(\c,\c^*) + (3 - \kappa_1) \lambda(r \vee \lambda_0) \right ).
       \]
       \end{thm}
       
       A direct implication of Theorem \ref{SML} is that $\ell(\cl,\c^*) \leq 4 \lambda ( I_w(\c^*) \vee \lambda_0)$. Hence, choosing $\lambda \sim \lambda_1(x)$ gives a convergence rate for $\ell(\cl,\c^*)$ of order $T(w)/\sqrt{n}$, up to a $\log(n)$ factor. If $T(w)$ is fixed, i.e. does not depend on $n$, this rate is roughly the same as the rate of convergence of the $k$-means codebook without margin assumption, as shown in $\cite{Biau08}$. 
       
       Besides, some asymptotic results for $\ell(\cl,\c^*)$ when both $d$ and $n$ are large may also be deduced from Theorem \ref{SML}, as stated by the following corollary.
       \begin{cor}\label{SMLAsymptotic}
       Let $\c^*$ be in $\mathcal{M}$ and denote by $d^*$ the quantity $|S(\c^*)|$. Assume that $\max_{p=1, \hdots, d}M_p =O(1)$, $n^{-1} \log(d) \rightarrow 0$, and $n^{-1} \lambda^{-2} \log(d \sqrt{n}) \rightarrow 0$.
       
       For $\hat{w}_p=1$, $\ell(\cl,\c^*) = O_P(\lambda d^*)$.
       
       For $\hat{w}_p = \hat{\sigma}_p$, if we further assume $\max_{p=1, \hdots, d} \sigma_p = O(1)$ and $1=O(\min_{p=1, \hdots, d} \sigma_p)$, then $\ell(\cl,\c^*) = O_P(\lambda d^*)$.             
       \end{cor} 
       This result may be compared for instance with Theorem 4.1 of \cite{Rigollet11}, in the framework of high dimensional regression. In this case an asymptotic convergence rate of $d^* \lambda$ may be similarly derived under the same assumptions (up to a $\log(n)$ factor) that $\log(d)n^{-1}$ $\rightarrow$  $0$ and $ \lambda^{-2} n^{-1} $ $\rightarrow$ $0$. This shows that the optimal distortion may be asymptotically attained for dimension $d$ of order $e^{n^{\kappa}}$, with $\kappa <1$, choosing $\lambda$ of order $n^{\frac{\kappa'-1}{2}}$, with $\kappa < \kappa' <1$. 
       
       Moreover, Corollary \ref{SMLAsymptotic} can provide a convergence rate  of order $O(d^* \log(d) n^{-1/2})$ for the excess distortion of these Lasso-type procedures, up to a $\log(n)$ factor, hence adapting the sparsity of the optimal codebooks. In comparison to the $O(d n^{-1/2})$ rate that can be derived for the excess distortion of the $k$-means codebook (see, e.g., \cite{Biau08}), this suggests that regularized $k$-means might outperform standard $k$-means whenever $d^* << d$ and $d$ is large. Some numerical illustration of this point is given below.
       
\textbf{Numerical illustration}: We consider the Gaussian mixture distributions with $4$ components, each of them having covariance matrix $I_d$ (identity matrix on $\mathbb{R}^d$), and with the following means:
\begin{align*}
\begin{array}{@{}ccccrcl}
\mu_1 &=& ( \overbrace{0.8, \hdots, 0.8}^5,\overbrace{-0.8,\hdots, -0.8}^5, \overbrace{0, \hdots, 0}^{d-10}),& \quad & \mu_3 & =& -\mu_1, \\
\mu_2 &=& (\overbrace{0.8, \hdots, 0.8}^{10}, \overbrace{0, \hdots, 0}^{d-10} ),& \quad &\mu_4 &=& - \mu_2.
\end{array} 
\end{align*}
      The weights of the mixture are chosen as $(0.3,0.2,0.2,0.3)$. For $d$ growing from $10$ to $500$, we compute the plain Lasso $k$-means codebooks with regularization parameter $\lambda(d)= 1.5 \times \log(d)/\sqrt{n}$, in the cases $n=50$ and $n=200$. Note that, since Gaussian mixture distributions have not a bounded support, this example does not fall in the scope of Theorem \ref{SML}. This issue might be bypassed considering truncated Gaussian mixture distributions, as exposed in Section \ref{Gaussianmixture}.  
      
      Following the approach of \textit{Algorithm 1} of \cite{Sun12}, the codebooks are computed using a Lloyd's-type algorithm: for any initial codebook, we update the assignments of data points to the closest code point and then update the code points to minimize the penalized squared distances to the previously assigned data points, using the Karush-Kuhn-Tucker condition that is necessary and sufficient when assignments are fixed. This procedure is repeated until convergence. Since every iteration decreases the penalized empirical distortion, the outcome of such an algorithm is clearly a local minimum of the penalized empirical distortion. This suggests that an effective global minimization of the penalized empirical distortion could be achieved by the comparison of the outcomes of several executions of the latter procedure with different initializations, as for the classical $k$-means implementation. 
      
      We choose as initial code points the ones given by the standard $k$-means algorithm, as suggested in \cite{Sun12}, and the means of the mixture, leading to few iterations before convergence, based on our limited experience. Then the best of these two codebooks in terms of penalized empirical distortion is chosen. In full generality, the choice of a good initialization for such an algorithm is likely to be a crucial issue, and is beyond the scope of this paper. It may also be noted that the choice of the constant $1.5$ is based on experimental observations. The calibration of such a constant might be more generally performed using cross-validation, as done in \cite{Sun12}. 
       
       Figure \ref{Lkmeansdistortions} below depicts the average distortions of both plain Lasso $k$-means and $k$-means codebook, over $100$ replications. The both panels show that the excess distortion of the $k$-means codebook grows linearly with respect to the dimension, whereas the excess distortion of the plain Lasso $k$-means codebook exhibits a dependence on the dimension that looks like sub logarithmic. In fact, in the case $n=50$, the Lasso $k$-means codebook turns out to be the zero codebook when $d$ is larger than about $100$, hence its constant excess distortion. 
\begin{figure}[h]
\includegraphics[width=0.8\textwidth, height=6cm]{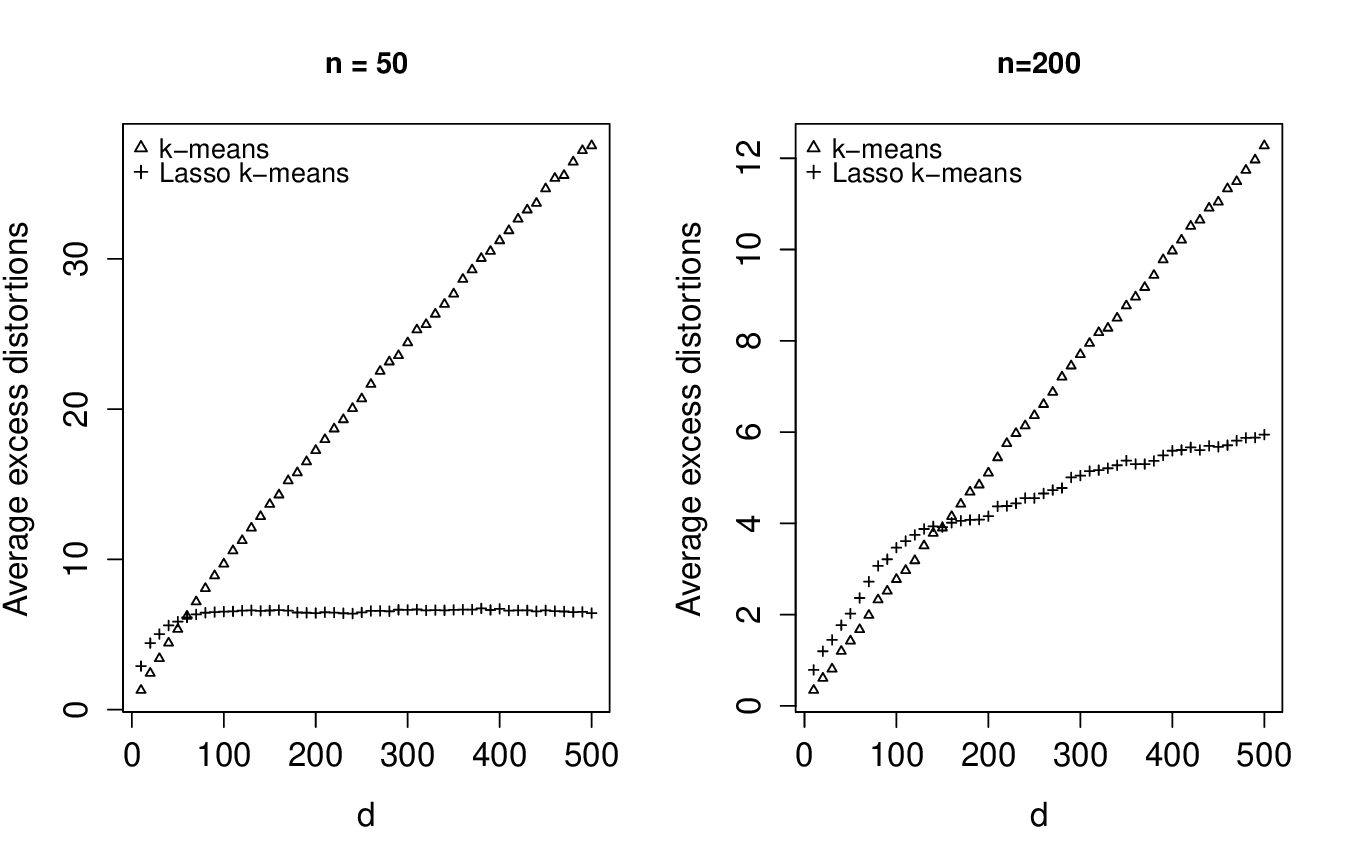}
\caption{\label{Lkmeansdistortions}Average excess distortions of the plain Lasso \texorpdfstring{$k$}{k}-means and \texorpdfstring{$k$}{k}-means codebooks over $100$ replications. }
\end{figure}

        Under the sole assumptions of Corollary \ref{SMLAsymptotic}, no results on the threshold Lasso may be stated, since Assumption \ref{WeightsAssumption} cannot be checked.
      
       \subsection{Convergence towards a sparse codebook and fast rate of convergence for the distortion}
          If $P$ satisfies Assumption \ref{margincondition} and Assumption \ref{nosubcodebook}, further results may be derived, following the approach of \cite{vandeGeer08}. To this aim, we defined, for a fixed codebook $\c^*$ and a weight family $w$, the set of $w$-sparse approximations of $\c^*$ at order $\lambda$ by   
       \begin{align*}\label{sparseapproximation}
       \mathcal{M}_{\lambda}(\c^*) = \left \{ \argmin_{S(\c) \subset S(\c^*)} {3R(\c) + 8 \kappa_0 \lambda^2 \|w_{S(\c)}\|^2} \right \},
       \end{align*}
       where $\kappa_0=\kappa'_0 \vee \kappa''_0$, as defined below Proposition \ref{TMC}. Then, the closest $w$-sparse approximation of $\c^*$ may be defined as $\c^*_{\lambda}(\c^*) \in \argmin_{\c \in \mathcal{M}_{\lambda}(\c^*)}{\| \c - \c^*\|}$. With a slight abuse of notation, the $w$-sparse approximation of a codebook $\c$ is defined by $\c_{\lambda}^*(\c) = \c_{\lambda}^*(\c^*(\c))$. It is immediate that, for the plain Lasso, $\|w_{S(\c)}\| = |S(\c)| = \|\c\|_0$, whereas for the normalized Lasso, $\|w_{S(\c)}\| = \sigma_{S(\c)}$. For the threshold Lasso, $\|w_{S(\c)}\|$ has the slightly more intricate expression
       \[
       \|w_{S(\c)}\|^2 = \frac{1}{\delta^2} \left | S(\c) \cap \{j| \| \c^{*,(j)} \| \leq \delta \} \right | + \sum_{S(\c) \cap \{j| \| \c^{*,(j)} \| > \delta \}} {\frac{1}{\| \c^{*,(j)} \|^2}}.
       \]
       However, it is easy to see that $\|w_{S(\c)}\| \leq |S(\c)|/\delta$.  If we assume that the optimal codebooks $\c^*$ are sparse, some guarantees on the support of the $\c^*_{\lambda}(\c^*)$'s may be given. To be more precise, the following subset of variables is introduced
        
        \[
        (S^+)^c  = \bigcap_{\c^* \in \mathcal{M}}{S(\c^*)^c}.
        \]
        $S^+$ may be thought of as the generalized support over optimal codebooks, extending the definition of these sets from the unique optimal codebook case. If $P$ has a unique optimal codebook, it is immediate that $S^+ = S(\c^*)$. However, even if all the codebook are sparse, $S^+$ may not be sparse. For instance, if $d=k=2$ and $P$ is a pointwise distribution with support $(-1,-1), (-1,1), (1,-1), (1,1)$ equally weighted, then every optimal codebook has at least one zero coordinate, whereas $S^+ = \{1,2\}$.
        
        From its definition, it is straightforward that $S(\c^*_{\lambda}(\c^*)) \subset S^+$, for $\c^*$ in $\mathcal{M}$. Nevertheless, in the case where $S^+ = \{1, \hdots, d \}$, the $\c^*_{\lambda}(\c^*)$'s may still have zero coordinates. In fact, $\c^*_{\lambda}(\c^*)$ may be thought of as tradeoff between distortion and size of the support, the latter being measured by $\|w\|_{S}^2$. As in the empirical case of Proposition \ref{KKTlasso}, this tradeoff property may be illustrated in the following way. 
       \begin{prop}\label{sparseapproximationsupport}
       Let $p$ be in $\{ 1, \hdots, d\}$. If
       \[
       \sigma_{p}^2 - R^*_p < \frac{8 \lambda^2 \kappa_0 w_p^2}{3},
       \]
       then, for every $\c^*$ in $\mathcal{M}$, 
       \[
       \c^{*, (p)}_{\lambda}(\c^*) = (0,\hdots, 0).
       \]
       \end{prop}
       The proof of Proposition \ref{sparseapproximationsupport} is given in Section \ref{ProofofPropositionsparseapproximationsupport}. Proposition \ref{sparseapproximationsupport}, as well as Proposition \ref{KKTlasso}, may be thought of as a comparison between the risk of  optimal codebooks and the risk of the codebook $\mathbf{0}$, on the $p$-th variable. It is worth noticing that, for the plain Lasso and threshold Lasso, the comparison is based on the difference $\sigma_{p}^2 - R^*_p$, whereas for the normalized Lasso only the ratio $R^*_p/\sigma_p^2$ is to be considered. Once more, this point suggests that the sparse $w$-approximation may not be sensitive to coordinate-wise dilations in this case. We are now in position to state sharper oracle results, on both the excess distortion and the convergence of the Lasso $k$-means codebooks.
       \begin{thm}\label{oracleinequality}
       Suppose that  Assumption \ref{nosubcodebook}, Assumption \ref{margincondition}  and Assumption \ref{WeightsAssumption} are satisfied. If 
       \[
       \lambda \geq \frac{2 \lambda_1(x)}{1 - \kappa_1} ,
       \]
        where $\lambda_1$ is defined in \eqref{minimalregularization}, then, with probability larger than $1 - r_1(n) - e^{-x}$, we have
        \begin{align}\label{Loracle}
        \ell(\cl,\c^*) + \lambda(1-\kappa_1) I_w(\cl-\c^*_{\lambda}(\cl)) \leq 3 \ell(\c^*_{\lambda},\c^*) + 8 \kappa_0 \lambda^2 \| w _{S(\c^*_{\lambda}(\cl))} \|^2   \vee 3 \lambda \lambda_0.
        \end{align}
       \end{thm}
          A consequence of Theorem \ref{oracleinequality} is $\ell(\cl,\c^*) \leq 8 \kappa_0 \lambda^2 \| w_{S(\c^*(\cl))} \|^2 \vee 3 \lambda\lambda_0$,  which provides an oracle inequality adapting the sparsity of the $\c^*(\cl)$'s. For instance, considering the plain Lasso, provided that $\c^* \neq \mathbf{0}$ for some $\c^*$, \eqref{Loracle} leads to $\ell(\cl,\c^*) \leq 8(\kappa_0 \vee 1) |S(\c^*(\cl)) | \lambda^2 \leq 8(\kappa_0 \vee 1) |S^+| \lambda^2$. However, Theorem \ref{oracleinequality} also deals with the case where the $\c^*$'s are not sparse, comparing the Lasso $k$-means codebook $\cl$ to the closest sparse $w$-approximations, for which Proposition \ref{sparseapproximationsupport} yields a reduced support whenever $\lambda$ is large enough.
         
           Theorem \ref{oracleinequality} may be considered as an application of Theorem 2.1 in \cite{vandeGeer13} to the $k$-means case, with a slight improvement in the analysis of the complexity term (see Section \ref{ProofofPropositiondeviationlocalisee} in the Appendix for more details). The numerical constants in Theorem \ref{oracleinequality} have been arbitrarily fixed for clarity sakeness, note however that a more general version of Theorem \ref{oracleinequality} can be derived the same way as Theorem 2.1 in \cite{vandeGeer13}. 
           
           At last, it is worth pointing out that the inequality $\ell(\cl,\c^*) \leq 8 \kappa_0 \lambda^2 \| w_{S(\c^*(\cl))} \|^2 \vee 3 \lambda\lambda_0$ provides a convergence rate in $1/n$, up to a $\log(n)$ factor, when $\lambda \sim \lambda_1$ and $w$ does not depend on $n$. Interestingly, this rate is the convergence rate of the $k$-means codebook $\hat{\c}_n$, when $P$ satisfies a margin condition, as described in \cite{Levrard14}.
           
           Similarly to Theorem \ref{SML}, Theorem \ref{oracleinequality} may provide asymptotic convergence rates for both distortion and distance to optimal codebooks.
           \begin{cor}\label{Asymptoticoracle}\
           
           For the plain and normalized Lasso, assume that $P$ satisfies Assumption \ref{nosubcodebook}, Assumption \ref{margincondition} so that $\kappa_0 = O(1)$, and the requirements of Corollary \ref{SMLAsymptotic}. If $n^{-1} \log(d) \rightarrow 0$ and  $n^{-1} \lambda^{-2} \log(d \sqrt{n}) \rightarrow 0$, then 
           \[
           \ell(\cl,\c^*) + \lambda(1-\kappa_1) I_w(\cl-\c^*_{\lambda}(\cl)) = O_P(|S^+| \lambda^2).
           \]

           For the threshold Lasso, assume that $\c^*$ is unique, $P$ satisfies Assumption \ref{margincondition} so that $\kappa_0=O(1)$, and $\max_{p=1, \hdots, d} M_p =O(1)$. If $\delta$  and $\lambda$ are chosen so that $\delta$ $\rightarrow 0$, $ d n^{-1} \delta^{-2}$ $\rightarrow 0$, and $\lambda^{-1} \log(n)^{1/2} n^{-1/2} $ $\rightarrow$ $0$, then 
           \[
           \ell(\cl,\c^*) + \lambda(1-\kappa_1) I_w(\cl-\c^*_{\lambda}(\c^*)) = O_P(d^* \lambda^2).
           \]

           As a consequence, in every case, $\| \cl - \c^*(\cl) \| = O_P(\sqrt{|S^+|} \lambda)$. 
           \end{cor}
           
            According to Theorem 3.1 of \cite{Levrard14}, the excess distortion of the $k$-means codebook is of order $O( d n^{-1})$, under the assumptions of Corollary \ref{Asymptoticoracle}. In comparison, Corollary \ref{Asymptoticoracle} yields convergence rates of order $O(|S^+| \log(d) n^{-1})$, up to a $\log(n)$ factor, for the plain and normalized Lasso. Again, this suggests that these procedures might outperform standard $k$-means procedures in terms of distortion in high dimensional settings, when $|S^+| << d$.

            The requirement $\kappa_0=O(1)$ may be thought of as an assumption on the local strong convexity of the excess distortion. This condition is similar to a uniform lower bound on the Hessian matrix of the excess distortion, as required for the asymptotic results in \cite{Sun12}. This asymptotic framework also allows for further comparison between our results and Theorem 1 of \cite{Sun12}, which states that, when choosing $\hat{w}_p = \| \hat{\c}_n^{(p)}\|^{-1}$, provided that $\c^*$ is unique, $n^{1/2} \lambda d$ $\rightarrow 0 $, and $n^{-2} \lambda^{-2} d $ $\rightarrow$ $0$, $\|\cl - \c^*\| = O_P(n^{1/2} \lambda d^{-1})$. 
           
          If we choose $\lambda$ close to the given lower bounds, for instance $\lambda= n^{-1/2} \log(d\sqrt{n})^{-1/2} u_n$ for our plain Lasso, and $\lambda = \sqrt{d}  n^{-1}u_n$ in the setting of \cite{Sun12}, with $u_n$ $\rightarrow$ $\infty$, then Theorem 1 of \cite{Sun12} yields $\| \cl - \c^* \| = O_P(u_n n^{-1/2}d^{-1/2})$, whereas Corollary \ref{Asymptoticoracle} gives a slightly worse bound, $\| \cl - \c^* \|= O_P(u_n \sqrt{|S^+|} \log(d \sqrt{n})n^{-1/2})$. However, Theorem 1 of \cite{Sun12} is valid for $d=o(\sqrt{n})$, when Corollary \ref{Asymptoticoracle} only requires $\log(d) = o(n)$ for the plain Lasso. 
            
            For the threshold Lasso, when $\c^*$ is unique, Corollary \ref{Asymptoticoracle} requires $d=o(n)$ and $\lambda=u_n \log(n)^{1/2} n^{-1/2}$ to get bounds on $\| \cl - \c^*\|$ of order $\sqrt{d^*}\log(n)^{1/2}n^{-1/2}u_n$, hence still worse than  $u_n n^{-1/2}d^{-1/2}$. It is interesting to note that these two very similar procedures give almost the same rate for $\| \cl - \c^*\|$, but with possibly very different choices of $\lambda$ (of order $n^{-1}$ in \cite{Sun12}, and of order $n^{-1/2}$ here). Some explanation can be given, noting that our results are intended to provide bounds on the distance between $\cl$ and $\c^*_{\lambda}(\c^*)$, where $\c^*_{\lambda}(\c^*)$ is possibly different from $\c^*$. Thus, the empirical processes involved in our derivations are not the same than those of \cite{Sun12}. Besides, Corollary \ref{Asymptoticoracle} states bounds in terms of the $I_w$ distance, rather than the Euclidean one. This $I_w$ distance is of particular interest, especially for coordinates which are not in $S(\c^*)$. For instance, if $j$ is not in $S(\c^*)$, and if we choose $\delta = d^{1/2}n^{-1/2} u_n$ along with  $\lambda=u_n \log(n)^{1/2} n^{-1/2}$, then Corollary \ref{Asymptoticoracle} ensures that $\| \cl^{(j)} \| = O_P(d^* \sqrt{d} \log(n)^{1/2} n^{-1} u_n^2)$. This convergence rate turns out to be faster in terms of $n$ than the one which can be derived from $\| \cl - \c^*\|$. In this particular case of a unique optimal codebook, further consistency results may be given, as described in the following subsection.
           
     \subsection{\texorpdfstring{Consistency of the threshold Lasso $k$-means}{Consistency of the threshold Lasso k-means}}       
             
        Throughout this subsection we assume that there exists a unique optimal codebook $\c^*$, up to relabeling. Let $j$ be in $S(\c^*)^c$. Then Theorem 2 in \cite{Sun12} established that $\cl^{(j)} \rightarrow 0$ in probability under strong assumptions on $P$. To be more precise, it is assumed in \cite{Sun12} that
        $P_{|V_j^*} = c_j^* + \varepsilon_j$, where $P_{|V_j^*}$ denotes the conditional law of $P$ on the optimal Voronoi cell centered at the $j$-th optimal code point $c_j^*$, $\varepsilon_j$ has independent coordinates, and the $\varepsilon_j$'s are independent. Theorem \ref{consistency} below gives a generalization of this result, along with a convergence rate for $\mathbb{P}( \cl^{(j)} \neq 0)$.
        
        \begin{thm}\label{consistency}
       Suppose that Assumption \ref{margincondition} is satisfied, and that $d$ is fixed. For $\hat{w}_p = 1/(\delta \vee \| \hat{\c}_n^{(p)} \|)$, if $ n^{-1} \log(n) \delta^{-2}$ $\rightarrow$ $0$ and $\delta$ $\rightarrow$ $0$, choose  $\lambda \sim \delta$.      
        Then, for every $j$ in $S(\c^*)^c$, we have
       \begin{align}\label{consistencyrate}
       \mathbb{P}\left ( \cl^{(j)} \neq 0 \right ) \underset{ n \rightarrow \infty}{=} O \left ( e^{-n \delta^2} \right ).
       \end{align}
       Moreover, for every $j$ in $S(\c^*)$, we have 
       \begin{align}\label{consistencyrate2}
       \mathbb{P}\left ( \cl^{(j)} = 0 \right ) \underset{ n \rightarrow \infty}{=} O \left ( e^{-n \delta^2 } \right ).
       \end{align} 
\end{thm}    
      Note that the two consistency rates given by \eqref{consistencyrate} and \eqref{consistencyrate2} are in fact of order $o(n^{-1})$. The choice $\lambda \sim \delta$ has been made to optimize the consistency rate. However, this choice may lead to suboptimal convergence rate for $\ell(\cl,\c^*)$ in Corollary \ref{Asymptoticoracle}. For instance, if we choose $\delta=n^{-\alpha}$, $0<\alpha<1/2$, then this choice of $\lambda$ leads to $\ell(\cl,\c^*) = O_P(n^{-2\alpha})$. In fact, we only need $\lambda^{-1} \log(n)^{1/2} n^{-1/2} $ $\rightarrow$ $0$, as in Corollary \ref{Asymptoticoracle}, to ensure that this model consistency result holds. Thus, the choice $\lambda \sim n^{-1/2}\log(n)^{(1/2+\varepsilon)}$, for a positive $\varepsilon$, provides both model consistency and almost optimal convergence of $\ell(\cl,\c^*)$. Some numerical illustration of this point is given below.
      
       The way Theorem \ref{consistency} is derived makes use of the Vapnik-Chervonenkis dimension of the Voronoi cells associated with the codebooks. Provided that a sharp bound on this dimension can be given, some asymptotic results when both $d$ and $n$ tend to infinity could also be stated.  
      
      \textbf{Numerical illustration}: To illustrate this consistency result, we consider the same Gaussian mixture distribution as in Section \ref{Slowrates}, but with fixed dimension $d=100$ and sample size $n$ growing from $100$ to $5000$. The threshold $\delta$ is chosen as $\delta(n) = n^{-1/6}$, and the threshold Lasso $k$-means codebooks are computed for two sequences of regularization parameters, namely $\lambda_{pred}(n)= 0.12 \times  \log(n)/ \sqrt{n}$ and $\lambda_{cons}(n) = 0.12 \times \delta(n)$. The choice of the constant $0.12$ is based on experimental observations and is clearly suboptimal for the small values of $n$, however it renders the comparison between the two regularization parameters easier.     
      \begin{figure}[h]     
      \includegraphics[width=0.75\textwidth, height=5.5cm]{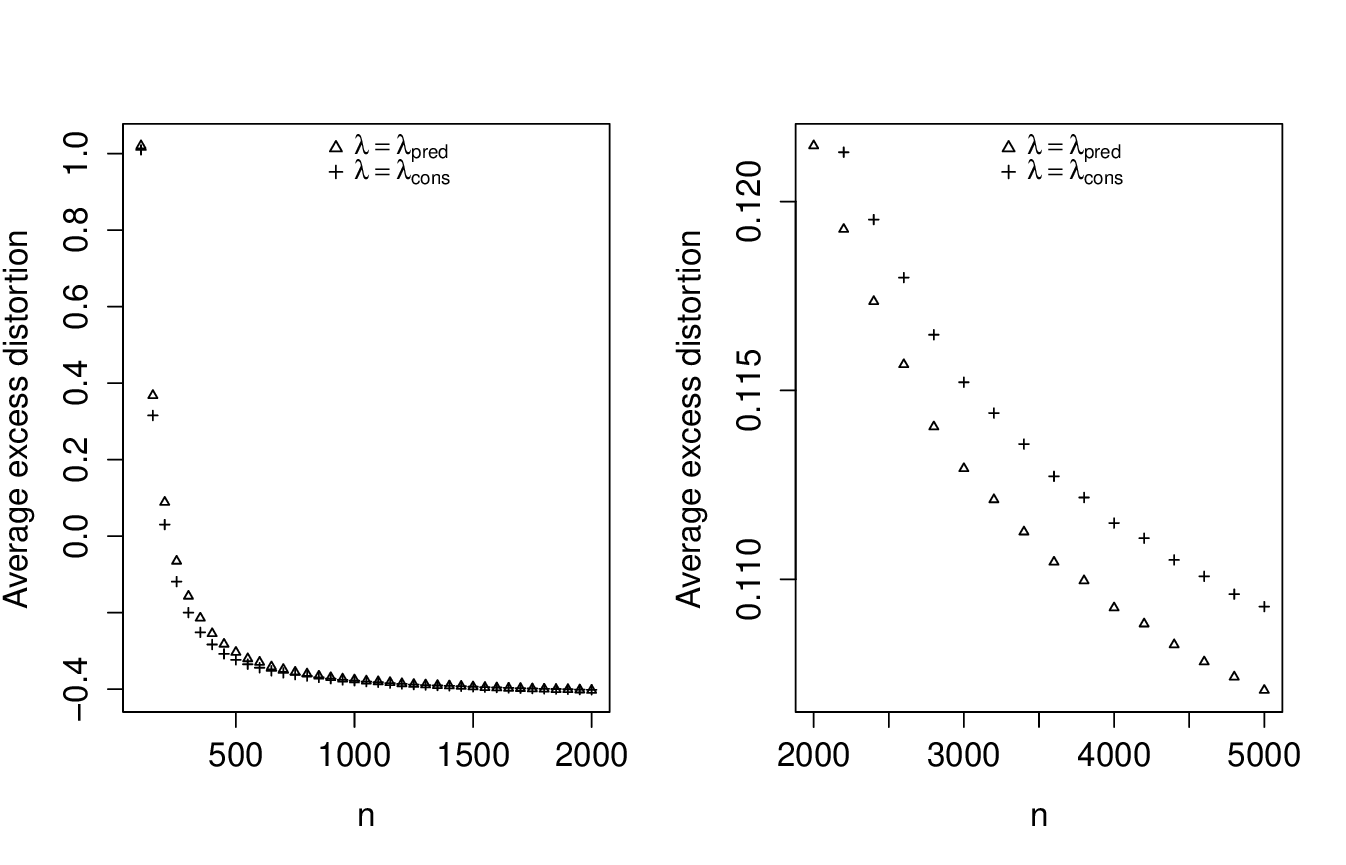}
      \caption{\label{PredConslosses}Average excess distortion for $\lambda \sim \lambda_1(\log(n))$ and $\lambda \sim \delta(n)$, over $100$ replications.}
      \end{figure}
      
      The right-hand side panel of Figure \ref{PredConslosses} shows that the threshold Lasso with parameter $\lambda_{pred}$ asymptotically outperforms the threshold Lasso with parameter $\lambda_{cons}$ in terms of distortion, as expected. However, for values of $n$ below $1000$, the $\lambda_{cons}$ regularization gives the best excess distortion, as shown by the left-hand side panel. The intuition behind Figure \ref{PredConslosses} is that for small values of $n$, $\lambda_{pred}$ under-penalizes irrelevant coordinates $p \geq 11$ compared to $\lambda_{cons}$, hence provides a too large support. On the other hand, for large values of $n$, the optimal support $\{1, \hdots, 10 \}$ is recovered by both regularization strategies, and in this case the milder penalization $\lambda_{pred}$ leads to better excess distortion. This intuition is confirmed by Figure \ref{PredConssupports} below.
      \begin{figure}[h]
      \includegraphics[width=0.75\textwidth, height=5.5cm]{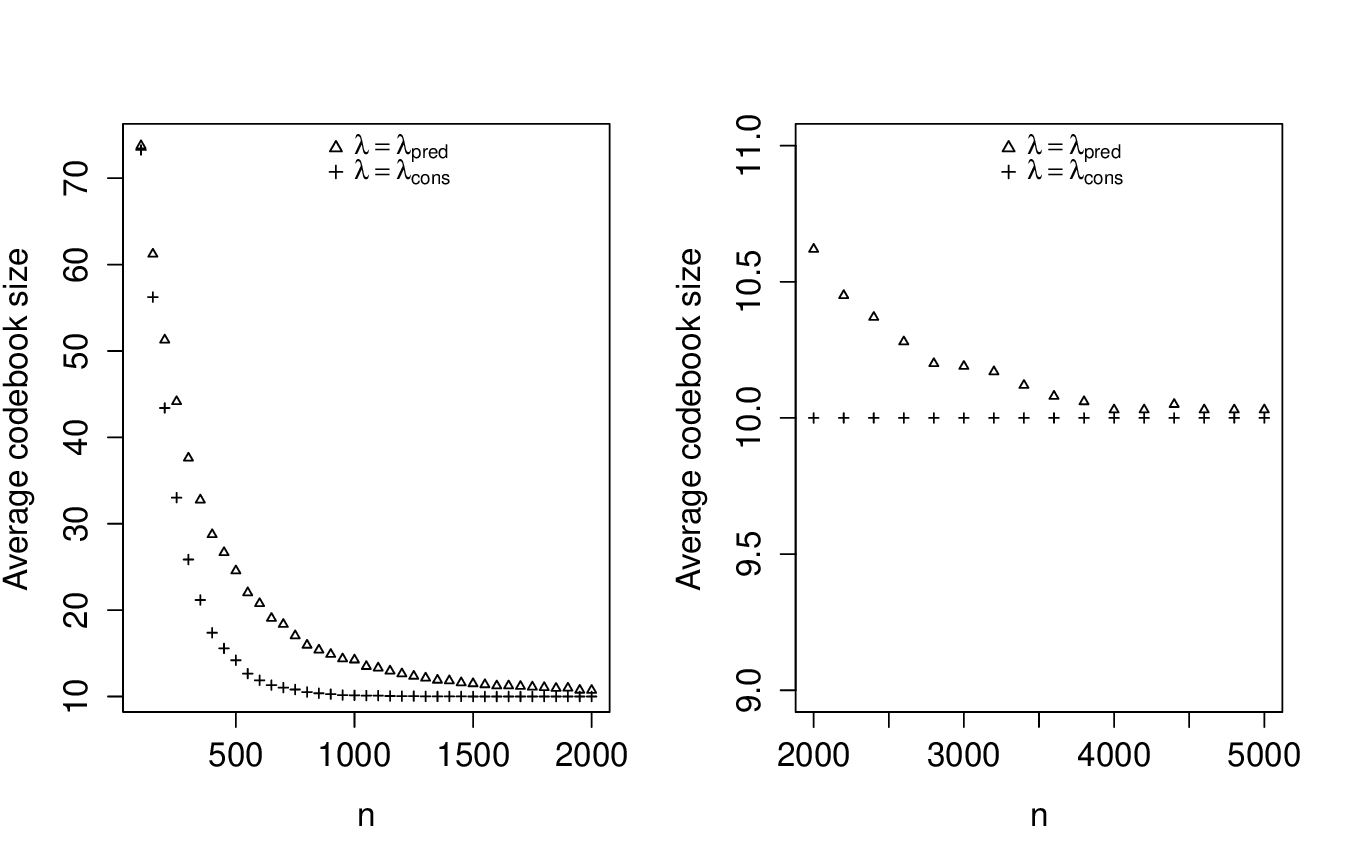}     
      \caption{\label{PredConssupports}Average support size for $\lambda \sim \lambda_1(\log(n))$ and $\lambda \sim \delta(n)$, over $100$ replications.}
      \end{figure}
      
      The left-hand side panel of Figure \ref{PredConssupports} illustrates the fact that the optimal support is more efficiently recovered through the $\lambda_{cons}$ regularization strategy. In the case $n \geq 2000$ corresponding to the right-hand side panel, the support of the threshold Lass $k$-means codebook with regularization parameter $\lambda_{cons}$ is exactly the optimal support in each of the replications, whereas the support corresponding to $\lambda_{pred}$ may sometimes contain some irrelevant coordinates. This illustrates that the probability of exact support recovery given by Theorem \ref{consistency} is in fact larger for the $\lambda_{cons}$ strategy than for the $\lambda_{pred}$ strategy.

      To apply Theorem \ref{consistency}, checking that $\c^*$ is unique remains a hard issue when $d>1$. The following section gives an example where this problem can be tackled using straightforward consequences of Assumption \ref{margincondition}. 
\subsection{Quasi-Gaussian mixture example}\label{Gaussianmixture}
  The aim of this section is to provide some theoretical background for the application of Theorem \ref{consistency} to the Gaussian mixture distributions. In general, a Gaussian mixture distribution $\tilde{P}$ may be defined by its density
                 \[
                 \tilde{f}(x) = \sum_{i=1}^{\tilde{k}}{\frac{\theta_i}{ (2 \pi)^{d/2} \sqrt{ \left | \Sigma_i \right | }}e^{-\frac{1}{2}(x-m_i)^t \Sigma_i^{-1} (x-m_i)}},
                 \]  
                 where $\tilde{k}$ denotes the number of components of the mixture, and the $\theta_i$'s denote the weights of the mixture, which satisfy $\sum_{i=1}^{k}{\theta_i} = 1$. Moreover, the $m_i$'s denote the means of the mixture, so that $m_i$ $\in$ $ \mathbb{R}^d$, and the $\Sigma_i$'s are the $d\times d$ variance matrices of the components. In this case, we define the active set $\tilde{S}$ as $S(\mathbf{m})$, where $\m$ denotes the codebook with code points the $m_i$'s. For convenience it is assumed that $\tilde{S} = \{1, \hdots, d' \}$, with $d'<d$.  We restrict ourselves to the case where the number of components $\tilde{k}$ is known, and match the size $k$ of the codebooks.
                 
                  Since the support of a Gaussian random variable is not bounded, we define the ``quasi-Gaussian" mixture model as follows, truncating each Gaussian component.
                 Let the density $f$ of the distribution $P$ be defined by
                 \[
                 f(x) = \sum_{i=1}^{k}{\frac{\theta_i}{ (2 \pi)^{d/2} N_i \sqrt{ \left | \Sigma_i \right | }}e^{-\frac{1}{2}(x-m_i)^t \Sigma_i^{-1} (x-m_i)}\mathbbm{1}_{\mathcal{B}(0,M)}},
                 \]
 where $N_i$ denotes a normalization constant for each Gaussian variable. To ensure that this model is close to the Gaussian mixture model, $M$ has to be chosen large enough, say $M \geq 2 \sup_{j=1, \hdots, k}{\| m_j \|}$ for instance.  Let $\sigma^2$ and $\sigma_{-}^2$ denote the largest and smallest eigenvalues of the $\Sigma_i$'s. Then the following proposition states that, provided that $\sigma$ is small enough, a model consistency result can be derived for the threshold Lasso.

        \begin{prop}\label{Gaussianconsistency}
        Assume that $\sigma_- \geq c^- \sigma$, for some constant $c^-$. Then there exists a constant $\sigma^+$ such that, if $\sigma \leq \sigma^+$, then Assumption \ref{margincondition} holds, and $\c^*$ is unique.
        
        Moreover, if $(\Sigma_i)_{pq}=0$, for every $(i,p,k)$ in $\{1,\hdots, k\} \times \{1, \hdots, d'\} \times \{d'+1, \hdots, d\}$, then, for every $j \geq d' +1$, the threshold Lasso with $n^{-1} \log(n) \delta^{-2} $ $\rightarrow$ $0$ satisfies 
        \[
        \mathbb{P}\left ( \cl^{(j)} \neq 0 \right ) \underset{ n \rightarrow \infty}{=} O \left ( e^{-n \delta^2} \right ).
        \]
        \end{prop} 
        The assumption on the covariance matrices requires that the variable in $\tilde{S}$ are independent from the variables in $\tilde{S}^c$. As shown in Section \ref{proofofgaussianconsistency}, this strong requirement ensures in fact that the support of the optimal codebook is contained in $\tilde{S}$. Proposition \ref{Gaussianconsistency} then directly follows from Theorem \ref{consistency}. 
        
        The first part of Proposition \ref{Gaussianconsistency} extends the results of Proposition 3.2 in \cite{Levrard14} to arbitrary dimension $d$. It also enhances the results of this Proposition, showing that there exists a unique optimal codebook instead of finitely many ones.
        
         It is worth noting that Proposition \ref{Gaussianconsistency} is valid when $k=\tilde{k}$. When $k$ differs from $\tilde{k}$, Assumption \ref{margincondition} may not be satisfied. For instance, if $P$ is rotationnaly symmetric and $k,d\geq 2$, then $\mathcal{M}$ cannot be finite, which contradicts Proposition \ref{TMC}.

\section{Proofs}\label{Proofs}
\subsection{Proof of Proposition \ref{KKTlasso}}

Let $W_1, \hdots, W_k$ be the Voronoi partition associated with $\cl$, and let $L(\cl)$ be the matrix of assignments, defined by
\[
L(\cl)_{i,j} = \mathbbm{1}_{X_i \in W_j}.
\]
Suppose that $\cl^{(p)} \neq 0$, where $\cl^{(p)}$ denotes the vector $(\hat{c}_{n,\lambda,1}^{(p)}, \hdots, \hat{c}_{n, \lambda, k}^{(p)})^{t}$, and denote by $X^{(p)}$ the vector $(X_1^{(p)}, \hdots, X_n^{(p)})^t$. Then the Karush-Kuhn-Tucker condition, for this penalized empirical risk minimization strategy, ensures that (see, e.g., the proof of Theorem 2 in \cite{Sun12})
\begin{align}\label{KKT}
\frac{-2}{\sqrt{n}} L(\cl)^t \left ( X^{(p)} - L(\cl) \cl^{(p)} \right ) + \sqrt{n} \lambda \frac{\hat{w}_p \cl^{(p)}}{\|\cl^{(p)}\|} = 0.
\end{align}  
Since $L(\cl)^t \left (X^{(p)} - L(\cl) \cl^{(p)} \right )$ is the following vector of size $k$ 
\[
\left ( \sum_{X_i \in W_1}(X_i^{(p)} - \hat{c}_{n,\lambda,1}^{(p)}), \hdots, \sum_{X_i \in W_k}(X_i^{(p)} - \hat{c}_{n,\lambda,k}^{(p)}) \right ),
\]
it may be noted that
\[
 \left \| L(\cl)^t \left ( X^{(p)} - L(\cl) \cl^{(p)} \right ) \right \|^2 = \sum_{j=1}^{k}{n_j^2(\bar{c}_j^{(p)} - \hat{c}_{n,\lambda,j}^{(p)})^2},
 \]
 where $n_j$ denotes the number of sample vectors $X_i$'s in $W_j$, and $\bar{c}_j$ denotes the empirical mean of the sample over the set $W_j$, that is $\bar{c}_j= \frac{1}{n_j} \sum_{X_i \in W_j}{X_i}$. Denote by $\hat{p}_j$ the empirical weight of $W_j$, i.e. $\hat{p}_j = n_j/n$, then
 \[
 \frac{1}{n^2}\left \| L(\cl)^t \left ( X^{(p)} - L(\cl) \cl^{(p)} \right ) \right \|^2 \leq \sum_{j=1}^{k}{\hat{p}_j(\bar{c}_j^{(p)} - \hat{c}_{n,\lambda,j}^{(p)})^2},
 \]
 where $\hat{p}_j \leq 1$ has been used.
 Let $Q_1$ be the quantizer which maps $W_j$ to $\bar{c}_j$, then it is easy to see that
 \[
 \sum_{j=1}^{k}{\hat{p}_j(\bar{c}_j^{(p)} - \hat{c}_{n,\lambda,j}^{(p)})^2} = \hat{R}_p(\cl) - \hat{R}_p(Q_1), 
 \]
 where we recall that $\hat{R}_p(Q) = P_n^{(p)} \|x - Q(x)\|^2$, for any quantizer $Q$.

 Denote by $\cl^{(-p)}$ the codebook that has $p$-th coordinate $0$ and the same coordinates as $\cl$ otherwise, and let $Q_2$ be the quantizer which maps $W_j$ to $\hat{c}_{n,\lambda,j}^{(-p)}$. Then, by definition, $\hat{R}(\cl) + I_{\hat{w}}(\cl) \leq \hat{R}(\cl^{(-p)}) + \lambda I_{\hat{w}}(\cl^{(-p)})$, and $\hat{R}(\cl^{(-p)}) \leq \hat{R}(Q_2)$. Thus, direct calculation leads to $\hat{\sigma}_p^2 \geq \hat{R}_p(\cl)$, according to \eqref{sousadditivité}. Since $\hat{R}_p(\cl) - \hat{R}_p(Q_1) \leq \hat{\sigma}_p^2 - \hat{R}_p$,  \eqref{KKT} ensures that
 \[
 \frac{\lambda \hat{w}_p}{2} \leq \sqrt{\hat{\sigma}_p^2 - \hat{R}_p}.
 \]
 
 \subsection{Proof of Proposition \ref{TMC}}\label{ProofofPropositionTMC}
       As mentioned below Proposition \ref{TMC}, a proof of the two first statements can be found in the proof of Proposition 2.2 in \cite{Levrard14}. The last statement follows from the compactness of $\left \{ \c \in C^k | S(\c) \subsetneq S(\c^*) \right \}$ and the fact that $\mathcal{M}$ is finite, knowing that $R()$ is continuous (see, e.g., Lemma 4.1 in \cite{Levrard14}). This ensures that $\inf_{\c^* \in \mathcal{M}} \inf_{S(\c) \subsetneq S(\c^*)} \ell(\c,\c^*) \geq c >0$, for some constant $c$, whenever Assumption \ref{nosubcodebook} is satisfied. Since $C^k$ is bounded, $\sup_{\c^* \in \mathcal{M}, S(\c) \subsetneq S(\c^*)}{\| \c - \c^*\|^2/ \ell(\c,\c^*)}$ is finite.
\subsection{Proof of Theorem \ref{SML}}            
 
   We recall here that $T(w)$ denotes the quantity $ T(w) = \max_{p=1, \hdots, d}{{M_p}/{w_p}}$.
   Let also $\bar{M}(w)$ be defined as $\sqrt{k}\|w\|^2T(w)$. It is immediate that, for every $\c$ in $C^k$, $I_w(\c) \leq \bar{M}(w)$. Moreover, we define $\bar{\gamma}$ by
  \[
  \bar{\gamma}(\c,x) = \min_{j=1, \hdots, k}{-2 \left\langle x, c_j \right\rangle + \|c_j\|^2},
  \]
  for every $\c$ in $C^k$ and $x$ in $\mathbb{R}^d$.
   The prediction results of this paper are based on the following concentration inequality, which connects the deviation of the empirical processes $(P-P_n)\bar{\gamma}(\c,.)$ to the $I_w$ norm. Note that, since $\bar{\gamma}$ is continuous and $(\mathbb{R}^d)^{k+1}$ is a separable metric space, the empirical processes introduced throughout the derivation are measurable.
\begin{prop}\label{localisation}
Suppose that $w$ is a deterministic sequence of weights. Denote by $u$ the quantity $\log\left ( \frac{\|w\|^2\sqrt{n}}{\sqrt{\log(kd)}}\right )$. If we denote by
\[
\lambda_0 = 8\sqrt{2 \pi} \sqrt{\frac{k\log(kd)}{n}} T(w),
\]
then, for every $x>0$, denoting by
\[
\lambda_1 = e \lambda_0 \left ( 1 + \sqrt{\frac{u+x}{k \log{kd}}} \right ),
\]
we have, for any fixed $\c'$ in $C^k$, with probability larger than $1-e^{-x}$,
\begin{align}\label{coreinequality}
\sup_{I_w(\c - \c') \leq 2 \bar{M}(w)}\frac{\left|(P-P_n)(\bar{\gamma}(\c,.) - \bar{\gamma}(\c',.)) \right |}{I_w(\c-\c') \vee \lambda_0} \leq \lambda_1.
\end{align} 
\end{prop}
  Proposition \ref{localisation} may be thought of as a slight generalization of inequality (7) in \cite{vandeGeer13}. Its proof, given in Section \ref{ProofofPropositionlocalisation} of the Appendix, differs from the original one by the use of Gaussian complexities rather than Talagrand's generic chaining principle to derive a multivariate contraction principle. We are now in position to prove Theorem \ref{SML}.
  
  We recall here that $\lambda \geq \frac{\lambda_1}{1-\kappa_1}$. From Assumption \ref{WeightsAssumption}, it easily follows that $(1-\kappa_1) I_w(\c) \leq I_{\hat{w}}(\c) \leq (1 + \kappa_1) I_w(\c)$, with probability larger than $1-r_1(n)$. On this event, we have, for every $\c$ in $C^k$, 
  \[
  P_n \bar{\gamma}(\cl,.) + \lambda (1-\kappa_1) I_w(\cl) \leq P_n \bar{\gamma}(\c,.) + \lambda(1+\kappa_1) I_w(\c). 
  \]
   Using \eqref{coreinequality}, with $\c'= \mathbf{0}$, it follows that
   \begin{multline}\label{baseinequalitySM}
      P \bar{\gamma} (\cl,.) \leq P \bar{\gamma}(\c,.) + \lambda(1+\kappa_1)I_w(\c)  + \lambda_1(I_w(\c)\vee \lambda_0) \\ + \lambda_1(I_w(\cl) \vee \lambda_0) - \lambda(1-\kappa_1) I_w(\cl).
   \end{multline}
   If $I_w(\cl) > \lambda_0$, adding $-P \bar{\gamma}(\c^*,.)$ on the both sides leads to 
   \[
   \ell(\cl,\c^*)  \leq \ell(\c,\c^*) + \lambda(1+\kappa_1) I_w(\c) + \lambda_1(I_w(\c) \vee \lambda_0).
   \]
   Hence
   \[
   \ell(\cl,\c^*) \leq \inf_{r>0} \inf_{I_w(\c) \leq r}{\ell(\c,\c^*) + 2 \lambda (r\vee \lambda_0)}.
   \]                  
  If $I_w(\cl) \leq \lambda_0$, \eqref{baseinequalitySM} may be written 
  \[
  \ell(\cl,\c^*) \leq \ell(\c,\c^*) + 2 \lambda (I_w(\c) \vee \lambda_0) + \lambda(1-\kappa_1) \lambda_0,
  \]
  hence
  \[
  \ell(\cl,\c^*) \leq \inf_{r>0} \inf_{I_w(\c) \leq r}{ \ell(\c,\c^*) + (3-\kappa_1) \lambda(r \vee \lambda_0)}.
  \]

\subsection{Proof of Proposition \ref{sparseapproximationsupport}}\label{ProofofPropositionsparseapproximationsupport}
 
 Let $S$ be a subset of $\{1, \hdots, d\}$, and let $p$ be in $S$ such that
 \[ 
 \sigma_p^2 - R_p^* < \frac{8 \kappa_0 \lambda^2 w_p^2}{3}.
 \]
 Denote by $\c^*_{S}$ an optimal codebook with support $S$, that is
 \[
 \c^*_{S} = \argmin_{ S(\c) = S} {R(\c)}.
 \]
 Then, according to \eqref{sousadditivité}, we may write
 \begin{align*}
 R(\c^*_{S\backslash\{p\}}) - R(\c^*_S) & \leq  R^*_{S\backslash\{p\}} + \sigma^2_{(S\backslash\{p\})^c} - (R^*_{S\backslash\{p\}} + R^*_p) - \sigma^2_{S^c} \\
                               & \leq \sigma_p^2 - R^*_p.
  \end{align*}
  Therefore
  \[
  3 R(\c^*_{S\backslash\{p\}}) + 8\lambda^2\kappa_0\|w_{S\backslash\{p\}}\|^2 < 3 R(\c^*_S) + 8\lambda^2\kappa_0\|w_S\|^2.
  \]

\subsection{Proof of Theorem \ref{oracleinequality}}

  Let $\c$ be a fixed codebook in $C^k$, and $\c'$ be another codebook in $C^k$. Following the notation of \cite{vandeGeer13}, with a slight abuse of notation, we denote by $I_{w,1}(\c' - \c)$ and $I_{w,2}(\c' - \c)$ the quantities
\begin{align*}
\left \{
\begin{array}{@{}ccc}
I_{w,1}(\c'-\c) &=& I_w( (\c'-\c)_{S(\c)} ), \\
I_{w,2}(\c'-\c) &=& I_w( (\c'-\c)_{S^c(\c)} ).
\end{array}
\right.
\end{align*} 
   The following result is derived from Lemma A.4 in \cite{vandeGeer08}.
   \begin{lem}\label{lemmevandeGeer}
   Let $\c$ and $\c'$ be in $C^k$, and denote by $\c^* = \c^*(\c')$. If $S(\c) \subsetneq S(\c^*)$ or $\c^*(\c) = \c^*$, for any $\delta >0$, we have
   \begin{align}\label{technicalsparsity}
   2 \lambda I_{w,1}(\c' - \c) \leq \frac{1}{\delta} \ell(\c,\c^*) + \frac{1}{\delta} \ell(\c',\c^*) + {2 \delta \kappa_0 \lambda^2}\|w_{S(\c)}\|^2.
   \end{align}
   \end{lem}
    
    The proof of Lemma \ref{lemmevandeGeer} can be found in \cite{vandeGeer08}. For the sake of completeness it is briefly recalled here.
    \begin{proof}[Proof of Lemma \ref{lemmevandeGeer}]
    Using Cauchy-Schwarz inequality, it is easy to see that
    \begin{align*}
    2\lambda I_{w,1}(\c'-\c) & \leq 2 \lambda\sqrt{\sum_{p \in S(\c)}{w_p^2}} \| \c' - \c \| \\
    &\leq 2 \lambda\sqrt{\sum_{p \in S(\c)}{w_p^2}}(\| \c' - \c^* \| + \| \c - \c^* \|). 
    \end{align*}
     Using the inequality $2ab \leq \kappa_0 \delta a^2 + \frac{1}{\delta \kappa_0}b^2$ and Proposition \ref{TMC} leads to
     \[
     2 \lambda  I_{w,1}(\c'-\c) \leq \frac{1}{\delta}(\ell(\c,\c^*) + \ell(\c',\c^*)) + 2 \delta \kappa_0 \lambda^2 \| w_{S(\c)} \|^2.
     \]           
      \end{proof}
   Equipped with this Lemma, we are in position to prove Theorem \ref{oracleinequality}. By definition, $\cl$ satisfies
   \[
   P_n \bar{\gamma}(\cl,.) + \lambda I_{\hat{w}} (\cl) \leq P_n \bar{\gamma}(\c,.) + \lambda I_{\hat{w}}(\c).
   \]
   Using Proposition \ref{localisation}, we get
   \begin{align}\label{baseinequality}
   P \bar{\gamma}(\cl,.) + \lambda I_{\hat{w},2}(\cl - \c) \leq P \bar{\gamma}(\c,.) + \lambda I_{\hat{w},1}(\cl - \c) + \lambda_1(\lambda_0 \vee I_w(\cl-\c)).
   \end{align}
   If $I_w(\cl - \c) > \lambda_0$, then adding $-P \bar{\gamma}(\c^*,.)$ on the both sides an using Assumption \ref{WeightsAssumption} leads to
   \[
   \ell(\cl,\c^*) + (1- \kappa_1) \lambda I_{w,2}(\cl-\c) \leq \ell(\c,\c^*) + (1 + \kappa_1) \lambda I_{w,1}(\cl - \c) + \lambda_1 I_w(\cl - \c).
   \]
   Hence, if $\c^*(\c) = \c^*(\cl) = \c^*$ or $S(\c) \subsetneq S(\c^*)$, 
   \begin{align*}
   \ell(\cl,\c^*) + \left [ (1-\kappa_1) \lambda - \lambda_1 \right ] I_w(\cl - \c) & \leq \ell(\c,\c^*) + 2 \lambda I_{w,1}(\cl - \c) \\
                              & \leq \frac{3}{2} \ell(\c,\c^*) + \frac{1}{2} \ell(\cl,\c^*) + 4 \kappa_0 \lambda^2 \| w_{S(\c)}\|^2,
   \end{align*}
   according to Lemma \ref{lemmevandeGeer}, with $\delta =2$. Noting that $\lambda_1 \leq (1-\kappa_1) \lambda /2$  yields
   \[
   \ell(\cl,\c^*) + (1-\kappa_1) \lambda I_w(\cl - \c) \leq 3 \ell(\c,\c^*) + 8 \kappa_0 \lambda^2 \| w_{S(\c)}\|^2. 
   \]
   If $I_w(\cl-\c) \leq \lambda_0$, then combining Assumption \ref{WeightsAssumption} with \eqref{baseinequality} entails
   \begin{align*}
   \ell(\cl,\c^*) + (1- \kappa_1) \lambda I_w(\cl-\c) & \leq \ell(\c,\c^*) + 2 \lambda I_{w,1}(\cl - \c) + \lambda_1 \lambda_0 \\
                                 & \leq \ell(\c,\c^*) + 3 \lambda \lambda_0.
   \end{align*}
   Since $\c^*_{\lambda}(\cl)$ satisfies $S(\c^*_{\lambda}(\cl)) \subsetneq S(\c^*(\cl))$ or $\c^*(\cl) = \c^*(\c^*_{\lambda}(\cl))$, choosing $\c=\c^*_{\lambda}$ gives the result.
   
  \subsection{Proofs of Corollary \ref{SMLAsymptotic} and Corollary \ref{Asymptoticoracle}}    
  For the plain Lasso, $\log(d)n^{-1}$ $\rightarrow$ $0$ and $\sup_{p=1, \hdots, d} M_p =O(1)$ ensures that $\lambda_0$ $\rightarrow$ $0$. $n^{-1} \lambda^{-2} \log(d\sqrt{n})$ yields $\lambda_1(\log(d\sqrt{n})) = O(\lambda)$. Applying Theorem \ref{SML} or \ref{oracleinequality} gives the results for $\ell(\cl,\c^*)$ and $\ell(\cl,\c^*) + \lambda(1-\kappa_1)I_w(\cl - \c^*_\lambda(\cl))$, since $\kappa_0 = O(1)$. The result on $\| \cl - \c^*(\cl)\|$ follows from Proposition \ref{TMC}.
  
  Similarly, for the normalized Lasso, the additional requirement $1=O(\min_{p=1, \hdots, d} \sigma_p)$ entails $\lambda_0$ $\rightarrow$ $0$, and Assumption \ref{WeightsAssumption} is satisfied, according to Proposition \ref{trueweights}. In turn, $\max_{p=1, \hdots, d} \sigma_p = O(1)$ leads to $\lambda_1(\log(d\sqrt{n})) = O(\lambda)$, hence the results. 
  
  At last, for the threshold Lasso, $\sup_{p=1, \hdots, d} M_p =O(1)$, $\delta$ $\rightarrow$ $0$ and $d n^{-1} \delta^{-2}$ imply that Assumption \ref{WeightsAssumption} is satisfied, $\lambda_0$ $\rightarrow$ $0$, and $d=O(n)$. Since $T(\delta) = O(1)$, combining $d=o(n)$ with $\lambda^{-1} \log(n)^{1/2} n^{-1/2}=o(1)$ leads to $\lambda_1(\log(n)) = o(\lambda)$. Noting that $\kappa_0 = O(1)$ and applying Theorem \ref{oracleinequality} gives the result.

  \subsection{Proof of Theorem \ref{consistency}}
  Let $j$ be in $S(\c^*)^c$, and denote by $f(n)$ the quantity $n\delta^2 \sim n \lambda^2$. Since $n^{-1} \delta^{-2} \log(n)$ $\rightarrow$ $0$, then $\log(n)=o(f(n))$. Hence Assumption \ref{WeightsAssumption} holds, for some $0<\kappa_1 <1$, with $r_1(n) = O(e^{-n\delta^2})$. Theorem \ref{consistency} follows from the Karush-Kuhn-Tucker condition, as in \cite{Sun12}, combined with some standard deviation bounds, which are listed below. Throughout this derivation, $K$ will denote a generic positive constant not depending on $n$.
  \begin{prop}\label{deviations}
  For every $x$ in $\mathbb{R}^d$, let $G^{(j)}(x,\c^*)$ denote the $k$ dimensional vector 
  \[
  \left ( x^{(j)} \mathbbm{1}_{W_1(\c^*)}(x), \hdots, x^{(j)} \mathbbm{1}_{W_k(\c^*)}(x) \right ).
  \]
  Then, we have
  \[
  \begin{array}{@{}ccc}
  \mathbb{P} \left [  \left \| (P - P_n) G^{(j)}(.,\c^*) \right \|  \geq K n^{- 1/2} f(n)^{1/2} \right ]  & = & O(e^{-f(n)}), \\
  \mathbb{P} \left [ \sup_{\c \in C^k}{|(P_n - P) \sum_{i \neq p}{\mathbbm{1}_{W_i(\c) \cap W_p(\c^*)}}}  |  \geq K n^{- 1/2} f(n)^{1/2} \right ] & = &O(e^{-f(n)}).
   \end{array}
   \]
  \end{prop}
  The proof of Proposition \ref{deviations} is deferred to Section \ref{ProofofPropositiondeviations}. Assume that $\cl^{(j)} \neq 0$, then  the K.K.T condition yields
  \[
  \frac{2}{\sqrt{n}} \| \hat{L} \hat{X}^{(j)} \| =  \left \| \sqrt{n} \lambda \hat{w}_j \frac{\cl^{(j)}}{\|\cl^{(j)}\|} + \frac{2}{\sqrt{n}} \hat{L}^t \hat{L} \cl^{(j)} \right \| \geq \sqrt{n} \lambda \hat{w}_j,
  \]
  since $\hat{L}^t \hat{L}$ is positive. According to Proposition \ref{trueweights}, it follows that $\sqrt{n} \lambda \hat{w}_j \geq (1-\kappa_1)n^{1/2}$, with probability larger than $1 - O(e^{-f(n)})$, when $n$ is large enough. On the other hand, we have
  \begin{align*}
  \| L(\cl) X^{(j)} \| & \leq \| L(\c^*) X^{(j)} \| + \| (L(\cl) - L(\c^*)) X^{(j)} \| \\
   & \leq n \| P_n G^{(j)} (.,\c^*) \| + M^{(j)} n P_n \sum_{i \neq p}{\mathbbm{1}_{W_i(\cl) \cap W_p(\c^*)}}.
  \end{align*}
  According to the centroid condition, $PG^{(j)} (.,\c^*) =0$. Thus, it follows from Proposition \ref{deviations} that $\| P_n G^{(j)} (.,\c^*) \| \leq K n^{-1/2} f(n)^{1/2}$ with probability $1 - O(e^{-f(n)})$. Lemma 4.2 in \cite{Levrard14} ensures that
  \[
  P \sum_{i \neq p}{\mathbbm{1}_{W_i(\cl) \cap W_p(\c^*)}} \leq p( K \| \cl - \c^* \|) \leq K \| \cl - \c^* \|,
  \]
  according to Assumption \ref{margincondition}. Besides, since $\log(n) = o(f(n))$, we may write $\lambda \sim \lambda_1(f(n))$. Taking into account that $\|w_{S(\c^*)} \|$ tends to $\sum_{i \in S(\c^*)}{1/\|\c^{*,(i)}\|}$, Theorem \ref{oracleinequality} yields $\ell(\cl,\c^*) \leq K \lambda^2$, with probability larger than $1 - O(e^{-f(n)})$, for $n$ large enough. On the same event,  Proposition \ref{TMC}  gives $\| \cl - \c^* \| \leq K \lambda$, which leads to 
  \[
   P_n \sum_{i \neq p}{\mathbbm{1}_{W_i(\cl) \cap W_p(\c^*)}} \leq K \lambda,
   \]
    with probability larger than $1 - O(e^{-f(n)})$, according to Proposition \ref{deviations}. Then the K.K.T condition entails
   \[
   (1- \kappa_1) n^{1/2} \leq K n^{1/2} \lambda, 
   \]
   with probability larger than $1-O(e^{-f(n)})$, which is impossible when $n$ is large enough.
   
   Conversely, if $j$ is in $S(\c^*)$ and $\cl^{(j)} = 0$, then the K.K.T condition yields
    \[
    \frac{2}{n}\| L(\cl) \hat{X}^{(j)} \| \leq \lambda \hat{w}_j.
    \]
    Since  $w_j$ tends to $1/\| \c^{*,(j)} \|$ and $\|PG^{(j)} (.,\c^*)\| >0$, according to the centroid condition, using the same concentration bounds as above leads to a contradiction.
   
 \subsection{Proof of Proposition \ref{Gaussianconsistency}}\label{proofofgaussianconsistency}
 The first part of Proposition \ref{Gaussianconsistency} is derived from the following Lemma, which extends Proposition 3.2 of \cite{Levrard14}. We denote by $\tilde{B}$ the quantity $\inf_{i \neq j}{\| m_i - m_j \|}$.
 \begin{lem}\label{Gaussiancalculus}
 Denote by $\eta = \sup_{j=1, \hdots, k}{1 - N_i}$. Then the risk $R(\m)$ may be bounded as follows.
 \begin{align}\label{meansrisk}
 R(\m) \leq \frac{\sigma^2 k \theta_{max} d }{(1-\eta)},
 \end{align}
 where $\theta_{max} = \max_{j=1, \hdots, k} {\theta_j}$. For any $0 < \tau <1/2$, let $\c$ be a codebook with a code point $c_i$ such that $\| c_i - m_j\| > \tau \tilde{B}$, for  every $j$ in $\{1, \hdots, k\}$. Then we have
 \begin{align}\label{risklowerbound}
 R(\c) > \frac{\tau^2 \tilde{B}^2 \theta_{min}}{4} \left ( 1 - \frac{2 \sigma \sqrt{d}}{\sqrt{2\pi } \tau \tilde{B}}e^{-\frac{\tau^2 \tilde{B}^2}{4 d \sigma^2}} \right )^d,
 \end{align}
 where $\theta_{min} = \min_{j=1, \hdots, k}{\theta_j}$. At last, if $\sigma^- \geq c_- \sigma$, for any $\tau'$ such that $2 \tau + \tau' < 1/2$, we have
 \begin{align}\label{gaussianweightfunction}
 \forall t \leq \tau' \tilde{B} \quad  p(t) \leq t \frac{2 k^2 \theta_{max} M^{d-1} S_{d-1}}{(2 \pi)^{d/2}(1- \eta) c_-^d \sigma^d} e^{- \frac{\left [ \frac{1}{2} - (2 \tau + \tau') \right ]^2 \tilde{B}^2}{2 \sigma^2}},  
 \end{align}
 where $S_{d-1}$ denotes the Lebesgue measure of the unit ball in $\mathbb{R}^{d-1}$.
 \end{lem}
       The proof of Lemma \ref{Gaussiancalculus} follows from direct calculation, as in the proof of Proposition 3.2 of \cite{Levrard14}. For the sake of completeness it is given in Section \ref{ProofofLemmaGaussiancalculus} of the Appendix. 
       
       Let $\tau$ and $\tau'$ be positive quantities satisfying $2 \tau + \tau' < 1/2$, and $\tau' > 8 \sqrt{2}M \tau/(1-2 \tau) \tilde{B}$. According to \eqref{meansrisk} and \eqref{risklowerbound}, if $\sigma$ is small enough, then every optimal codebook $\c^*$ satisfies $\sup_{j=1, \hdots, k} {\|m_j - c^*_j\|} \leq \tau \tilde{B}$, up to relabeling code points. 
        
        On the other hand, \eqref{gaussianweightfunction} ensures that, for $\sigma$ small enough, $P$ satisfies Assumption \ref{margincondition} with radius $r_0 \geq \tau' \tilde{B}$. Let $\c^*$ be an optimal codebook. According to $i)$ of Proposition 2.2 in \cite{Levrard14}, no other optimal codebook can be found in a ball of radius $(1-2 \tau) \tilde{B} \tau'/4 \sqrt{2} M$ centered at $\c^*$. Since  $(1-2 \tau) \tilde{B} \tau'/4 \sqrt{2} M > 2 \tau$, this proves that $\c^*$ must be unique.

         To apply Theorem \ref{consistency}, we need to show that $S(\c^*) \subset \tilde{S}$. Suppose that there exists $j \geq d' +1$ such that $\c^{*,(j)} \neq 0$. Let $s$ denote the orthogonal transformation defined by $s(x_1, \hdots, x_{d'},x_{d'+1}, \hdots, x_d) = (x_1, \hdots, x_{d'}, - x_{d' +1}, \hdots, -x_{d})$. Since $(\Sigma_{i})_{p,q} = 0$, for every $(i,p,k)$ in $\{1,\hdots, k\} \times \{1, \hdots, d'\} \times \{d'+1, \hdots, d\}$,  $P$ is invariant through composition by $s$. Hence $s(\c^*)$ is an optimal codebook, and $ \c^* \neq s(\c^*)$, which contradicts the fact that $\c^*$ is unique.
         
         
          \section*{Acknowledgement}
          
          The author is supported by the ANR project TopData ANR-13-BS01-0008.
          
  \section{Appendix}\label{Appendix}
  \subsection{Proof of Proposition \ref{trueweights}}\label{proofofpropositiontrueweights}
  A bounded difference inequality such as Theorem 6.2 of \cite{Massart13} yields, for $x >0$,
  \[
  \mathbb{P} \left (  \max_{p=1, \hdots, d} \left | \frac{\hat{\sigma}_p^2}{\sigma_p^2} -1 \right | \geq \mathbb{E} \left [ \max_{p=1, \hdots, d} \left | \frac{\hat{\sigma}_p^2}{\sigma_p^2} -1 \right | \right ] + x \right ) \leq e^{-\frac{2nx^2}{T^{4}(\sigma)}}.
  \]
   Besides, Theorem 2.8 of \cite{Massart13} ensures that, for every $p=1, \hdots, d$, $\hat{\sigma}_p^2/\sigma_p^2 -1$ is a subgaussian random variable with variance bounded by $T(\sigma)^4/4n$. For a comprehensive introduction to subgaussian random variables and its application to empirical process theory, the interested reader is referred to Section 2.3 of \cite{Massart13} or Chapter 14 of \cite{VandeGeer11}. A direct application of Theorem 2.5 of \cite{Massart13} leads to
   \[
               \mathbb{E} \left [ \max_{p=1, \hdots, d} \left | \frac{\hat{\sigma}_p^2}{\sigma_p^2} -1 \right | \right ] \leq \frac{T(\sigma)^2\sqrt{\log(d)}}{\sqrt{2n}},
               \]
               hence
                \[
  \mathbb{P} \left (  \max_{p=1, \hdots, d} \left | \frac{\hat{\sigma}_p^2}{\sigma_p^2} -1 \right | \geq 
  \frac{T(\sigma)^2\sqrt{\log(d)}}{\sqrt{2n}} \left [ 1 + \sqrt{\frac{x}{\log(d)}} \right ] \right ) \leq e^{-x}.
  \]
   Since $\sqrt{2n} \kappa_1 > T(\sigma)^2 \sqrt{\log(d)}$, choosing $x= \log(d) \left ( \frac{\sqrt{2n} \kappa_1}{T(\sigma)^2 \sqrt{\log(d)}}-1 \right )^2$ leads to the result.
   
   For the threshold Lasso with unique optimal codebook, Theorem 3.1 in \cite{Levrard14}, combined with Assumption \ref{margincondition}, provides a constant $C_0$ such that
   \[
   \| \hat{\c}_n - \c^* \| \leq C_0 M \sqrt{\frac{k }{n} \left ( 1 + \frac{x}{k } \right )}, 
   \]
   with probability larger than $1-e^{-x}$. Since, for every $p$ in $\left \{1, \hdots, d \right \}$, 
   \[
   \left | \frac{\| \c^{*,(p)}\| \vee \delta}{\| \hat{\c}_n ^{(p)} \| \vee \delta} - 1 \right | \leq \frac{\left |\| \c^{*,(p)}\| \vee \delta - \| \hat{\c}_n^{(p)} \| \vee \delta \right |}{\delta} \leq \frac{\| \hat{\c}_n - \c^* \|}{\delta},
   \]
  the results easily follows.

  \subsection{Proof of Proposition \ref{localisation}}\label{ProofofPropositionlocalisation}
  
  For a fixed $\c$ in $\c^k$, denote by $Z_r(\c)$ the following random variable
  \[
  Z_r(\c) = \sup_{I_w(\c' - \c) \leq r} \left | (P-P_n)(\bar{\gamma}(\c',.) - \bar{\gamma}(\c,.) )\right |.
  \]
  The following proposition gives a non-asymptotic bound on $Z_r(\c)$.
  \begin{prop}\label{deviationlocalisee}
  Suppose that $kd>1$ and $w$ is deterministic. Let $x>0$, and $\c$ be a fixed codebook. Then, with probability larger than $1-e^{-x}$,
  \[
  Z_r(\c) \leq 8\sqrt{2 \pi} \sqrt{\frac{k\log(kd)}{n}}r T(w) \left (1 + \frac{1}{4 \sqrt{\pi}}\sqrt{\frac{x}{k\log(kd)}} \right ).
  \]
  \end{prop}
  The proof of Proposition \ref{deviationlocalisee} is postponed to the next subsection. Proposition \ref{localisation} derives from a peeling argument, as in Section 3.4 of \cite{vandeGeer13}, combined with Proposition \ref{deviationlocalisee}.
  Let $a$ be the smallest integer such that $e^{-(a-1)} 2 \bar{M}(w) \leq \lambda_0$, and take $u_0 = \log(a)$ (we recall here that $\bar{M}(w)=\sqrt{k}\|w\|^2T(w)$ is an upper bound on $I_w(\c)$, for $\c$ in $C^k$). Then it is easy to see that
  $u_0 \leq u$, where $u$ is defined in Proposition \ref{localisation}. We may write
  \begin{align*}
  \mathbb{P} & \left ( \sup_{I_w(\c' - \c) \leq 2\bar{M}(w)} \frac{\left |(P-P_n)(\bar{\gamma}(\c',.) - \bar{\gamma}(\c,.))\right |}{I_w(\c'-\c) \vee \lambda_0} \geq \lambda_1 \right ) \\ 
  & \leq \begin{aligned}[t]
   \sum_{j=1}^{a-1} & \mathbb{P} \left (  \sup_{ \substack{I_w(\c' - \c) \leq  2 e^{-(j-1)}\bar{M}(w) \\ I_w(\c' - \c) \geq 2 e^{-j}\bar{M}(w)}} \frac{\left |(P-P_n)(\bar{\gamma}(\c',.) - \bar{\gamma}(\c,.)) \right |}{2e^{-j}\bar{M}(w)} \geq \lambda_1 \right ) \\
   & + \mathbb{P} \left ( \sup_{I_w(\c' - \c) \leq 2e^{-(a-1)}\bar{M}(w)} \frac{\left |(P-P_n)(\bar{\gamma}(\c',.) - \bar{\gamma}(\c,.)) \right |}{ 2e^{-a}\bar{M}(w)} \geq \lambda_1 \right )
   \end{aligned} \\
   & \leq \sum_{j=1}^{a} \mathbb{P} \left ( Z_{2e^{-(j-1)}\bar{M}(w)} \geq 2e^{-(j-1)}\bar{M}(w) \lambda_0 \left ( 1 + \sqrt{\frac{u+x}{k \log{kd}}} \right ) \right) \\
   & \leq ae^{-u}e^{-x},
   \end{align*}
   where the last inequality follows from Proposition \ref{deviationlocalisee}. Noticing that $ae^{-u} \leq 1$ proves the result.
  \subsection{Proof of Proposition \ref{deviationlocalisee}}\label{ProofofPropositiondeviationlocalisee}
  This proof is a slight modification of the proof of Theorem 3.1 in \cite{Levrard14}, and mainly relies on the use of Gaussian complexities combined with Slepian's Lemma (see, e.g., Theorem 13.3 in \cite{Massart13}).  
  For every $j=1, \hdots, k$, if $I_w(\c' - \c) \leq r$, then, for all $x$ in $C$, using $|a^2 - b^2| \leq 2 \max(|a|,|b|)|a-b|$ for $a$,$b$ in $\mathbb{R}$ and the Cauchy-Schwarz inequality, we get
  \begin{align*}
  \left | -2 \left\langle x , c'_j \right\rangle + \|c'_j\|^2 + 2 \left\langle x , c_j \right\rangle - \|c_j\|^2 \right | 
  & \leq 2 \sum_{p=1}^{d} |x|_p \left | c_j^{\prime (p)} - c_j^{(p)} \right | + \sum_{p=1}^{d} \left | {c_j^{\prime (p)}}^2 - {c_j^{(p)}}^2 \right | \\
  & \leq 4 \sum_{p=1}^{d} \frac{M_p}{w_p} w_p \sum_{l=1}^{k} \mathbbm{1}(l=j) \left | c_l^{\prime (p)} - c_l^{(p)} \right | \\
  & \leq 4 T(w) I_w(\c' - \c) \leq 4r T(w).
  \end{align*}
  which leads to
  \[
  \| {\bar{\gamma}}(\c',.) - {\bar{\gamma}}(\c,.) \|_{\infty} := \sup_{ x \in C}{\| {\bar{\gamma}}(\c',x) - {\bar{\gamma}}(\c,x) \|}  \leq 4rT(w).
  \]
  As a consequence, a bounded difference concentration inequality (see, e.g., Theorem 6.2 in \cite{Massart13}) yields, with probability larger than $1-e^{-x}$,
  \[
  Z_r(\c) \leq \mathbb{E} Z_r(\c) + 4rT(w)\sqrt{\frac{2x}{n}}.
  \]
  It remains to bound from above $\mathbb{E} Z_r(\c)$. According to the symmetrization principle (see, e.g., Section 2.2 of \cite{Koltchinskii04}), introducing some independent Rademacher variables $\varepsilon_i$ and standard Gaussian variables $g_i$ (also independent of the $\varepsilon_i$'s), we have
  \begin{align*}
  \mathbb{E} Z_r(\c) & \leq 2 \mathbb{E}_X\mathbb{E}_{\varepsilon} \sup_{I_w(\c' - \c) \leq r}{\frac{1}{n} \sum_{i=1}^{n}{\varepsilon_i ({\bar{\gamma}}(\c',X_i) - {\bar{\gamma}}(\c,X_i))}} \\
  & = 2 \sqrt{\frac{\pi}{2}} \mathbb{E}_X\mathbb{E}_{\varepsilon} \sup_{I_w(\c' - \c) \leq r} \mathbb{E}_g \left [\frac{1}{n} \sum_{i=1}^{n}{\varepsilon_i |g_i |  ({\bar{\gamma}}(\c',X_i) - {\bar{\gamma}}(\c,X_i))} \right ] \\
&\leq 2 \sqrt{\frac{\pi}{2}} \mathbb{E}_X\mathbb{E}_{g} \sup_{I_w(\c' - \c) \leq r}{\frac{1}{n} \sum_{i=1}^{n}{g_i ({\bar{\gamma}}(\c',X_i) - {\bar{\gamma}}(\c,X_i))}}.
  \end{align*}
Let $\c$ and $X_1, \hdots, X_n$ be fixed, and define, for $\c'$ such that $I_w(\c'-\c) \leq r$ the Gaussian process
  \[  
  Y_{\c'} =\sum_{i=1}^{n}g_i ({\bar{\gamma}}(\c',X_i)-\bar{\gamma}(\c,X_i)) .
  \]
  Since, for every codebooks $\c'_1$ and $\c'_2$,
  \[ 
  ({\bar{\gamma}}(\c'_1,X_i) - {\bar{\gamma}}(\c'_2,X_i))^2 \leq \max_{j=1, \hdots, k}{8 \left\langle c'_{1,j} -c'_{2,j},X_i \right\rangle^2 + 2 (\|c'_{1,j}\|^2 - \|c'_{2,j}|^2)^2},
  \]
  it is easy to see that
  \[
  \Var(Y_{\c'_1} - Y_{\c'_2}) \leq \sum_{i=1}^{n}\sum_{j=1}^{k}8 \left\langle c'_{1,j} -c'_{2,j},X_i \right\rangle^2 + 2n\sum_{j=1}^{k} (\|c'_{1,j}\|^2 - \|c'_{2,j}\|^2)^2.
  \]
  To derive bounds on the Gaussian complexity defined above, the following comparison result between Gaussian processes is needed.
  \begin{thm}[Slepian's Lemma]\label{Slepian}
						    Let $\mathcal{V}$ be a separable metric space, and $Y_t$, $N_t$, $t$ in $\mathcal{V}$, be some continuous centered real Gaussian processes. Assume that
						    \[
						    \forall t_1, t_2 \in \mathcal{V} \quad \Var(Y_{{t_1}} - Y_{{t_2}}) \leq \Var(N_{{t_1}} - N_{{t_2}}),
						    \]
							then
							\[
							\mathbb{E} \sup_{t \in \mathcal{V}} Y_t \leq  \mathbb{E} \sup_{t \in \mathcal{V}} N_t. 
							\]
							\end{thm}
Theorem \ref{Slepian} is a straightforward extension of Theorem 13.3 in \cite{Massart13} to separable index sets. Denote by $\mathcal{V}$ the separable set of codebooks $\c'$ in $C^k$ such that $I_w(\c' - \c) \leq r$. Now introduce, for $\c'$ such that $I_w(\c'-\c) \leq r$, the following Gaussian process
    \[
    N_{\c'} = 2\sqrt{2}\sum_{i=1}^{n}\sum_{j=1}^{k} \left\langle c_j' - c_j ,X_i\right\rangle \xi_{i,j} + \sqrt{2n} \sum_{j=1}^{k} (\| c'_j \|^2 - \| c_j\|^2) \xi'_{j},
    \]
    where the $\xi$'s and $\xi'$'s are independent standard Gaussian random variables. Note that $\c' \mapsto N_{\c'}$ is continuous, and for all $\c'_1$ and $\c'_2$ in $\mathcal{V}$, $ \Var(Y_{\c'_1}- Y_{\c'_2}) \leq \Var(N_{\c'_1} - N_{\c'_2})$. Consequently, applying Theorem \ref{Slepian}  yields
    \[
    \mathbb{E}_{g} \sup_{I_w(\c' - \c) \leq r}{Y_{\c'}} \leq  \mathbb{E}_{\xi,\xi'} \sup_{I_w(\c' - \c) \leq r}{N_{\c'}}.
    \]
    It follows that
    \begin{multline*}
    \mathbb{E}_{\xi,\xi'} \sup_{I_w(\c' - \c) \leq r}{N_{\c'}} \leq \mathbb{E}_{\xi} \sup_{I_w(\c'-\c)\leq r}{2\sqrt{2} \sum_{i=1}^{n}\sum_{j=1}^{k} \left\langle c'_j - c_j,X_i\right\rangle \xi_{i,j}} \\
    + \mathbb{E}_{\xi'}\sup_{I_w(\c'-\c) \leq r}{\sqrt{2n} \sum_{j=1}^{k} (\| c'_j\|^2 - \| c_j\|^2) \xi'_j}.
    \end{multline*}
    The first term of the right side can be bounded as follows.
    \begin{align*}
    &\mathbb{E}_{\xi} \sup_{I_w(\c'-\c)\leq r}{2\sqrt{2} \sum_{i=1}^{n}\sum_{j=1}^{k} \left\langle c'_j - c_j,X_i\right\rangle \xi_{i,j}} \\
    & \leq 2\sqrt{2} \mathbb{E}_{\xi} \sup_{I_w(\c'-\c)\leq r}{\sum_{j=1}^{k}\left\langle c'_j - c_j, \sum_{i=1}^{n}{\xi_{i,j} X_i} \right\rangle} \\
    & \leq 2\sqrt{2} \mathbb{E}_{\xi} \sup_{I_w(\c'-\c)\leq r}{ \left (\sum_{j=1}^{k}\sum_{p=1}^{d} w_p | {c'}_j^{(p)} - {c}^{(p)}_j| \right ) \max_{j, p}{\left | \sum_{i=1}^{n} \frac{\xi_{i,j}X_i^{(p)}}{w_p} \right |}} \\
    & \leq 2 \sqrt{2k}r \mathbb{E}_{\xi} \max_{j=1, \hdots,k , p=1, \hdots, d}{\left | \sum_{i=1}^{n} \frac{\xi_{i,j}X_i^{(p)}}{w_p} \right |}.
    \end{align*}
    Note that, for every $(j,p)$, the random variable $\sum_{i=1}^{n} \frac{\xi_{i,j}X_i^{(p)}}{w_p}$ is Gaussian, with variance bounded by $n T^2(w)$. Consequently, applying Theorem 2.5 in \cite{Massart13} gives
    \[
    \mathbb{E}_{\xi}\max_{j=1, \hdots,k , p=1, \hdots, d}{\left | \sum_{i=1}^{n} \frac{\xi_{i,j}X_i^{(p)}}{w_p} \right |} \leq T(w) \sqrt{2n \log(kd)}.
    \]   
    In turn, the second term of the right side may be bounded by
    \begin{align*}
    \mathbb{E}_{\xi'}\sup_{I_w(\c'-\c) \leq r}&{\sqrt{2n} \sum_{j=1}^{k}  (\| c'_j\|^2 - \| c_j\|^2) \xi'_j} \\
     & \leq \sqrt{2n} \mathbb{E}_{\xi'}\sup_{I_w(\c'-\c) \leq r} \sum_{j=1}^{k}\left ( \sum_{p=1}^d w_p | {c'}_j^{(p)} - {c}^{(p)}_j| \frac{2M_p}{w_p} \right )\left | \xi'_j \right | \\
    & \leq 2\sqrt{2n} T(w) \mathbb{E}_{\xi'}\sup_{I_w(\c'-\c) \leq r} I(\c'-\c) \sqrt{\sum_{j=1}^k {\xi'}^2_j} \\
    & \leq 2T(w)r \sqrt{2nk}.
    \end{align*} 
    Combining these two bounds leads to 
    \[
    \mathbb{E} Z_r(\c) \leq 8\sqrt{2 \pi} \sqrt{\frac{k\log(kd)}{n}}r T(w).
    \]
  
  \subsection{Proof of Proposition \ref{deviations}}\label{ProofofPropositiondeviations}  
    For every $u$ in $\mathbb{R}^k$, let us denote by $Y_u$ the function $\left\langle u , G^{(j)}(.,\c^*) \right\rangle$, so that $\|(P-P_n)G^{(j)}(.,\c^*)\|= \sup_{\|u\| \leq 1}{(P-P_n)Y_u}:=Y$. Since, for every $u$ such that $\|u\| \leq 1$ and for every $x$ in $C$, $|Y_u(x)| \leq M_j$, a bounded difference inequality yields (see, e.g., Theorem 6.2 in \cite{Massart13}), with probability larger than $1-e^{-x}$,
    \[
    Y \leq \mathbb{E}Y + M_j\sqrt{\frac{2x}{n}}.
     \]
   An upper bound on $\mathbb{E}Y$ may be derived the same way as in the proof of Proposition \ref{deviationlocalisee}: introducing some Rademacher random variables $\varepsilon_i$, $i=1, \hdots, n$, and  using the symmetrization principle, we get
   \begin{align*}
   \mathbb{E}Y & \leq 2 \mathbb{E}_{X,\varepsilon} \sup_{\| u \| \leq 1}{\left\langle u , \frac{1}{n} \sum_{i=1}^{n}{\varepsilon_i G^{(j)}(X_i,\c^*)} \right\rangle} \\
         & \leq \frac{2}{n} \sqrt{\mathbb{E}_{X,\varepsilon}\left \| \sum_{i=1}^{n}{\varepsilon_i G^{(j)}(X_i,\c^*)} \right \|^2} \\      
         & \leq \frac{2 M_j}{\sqrt{n}},
         \end{align*}
         according to Cauchy-Schwarz and Jensen's inequalities. Choosing $x = f(n)$ gives the first concentration inequality of Proposition \ref{deviations}.
         
         Now consider the $\{0,1\}$-valued random variables $Y_{\c}$, indexed by $C^k$, defined by $Y_{\c}= \mathbbm{1}_{\bigcup_{i \neq p} W_i(\c) \cap W_p(\c^*)}$. According to \cite{Pollard82}, the sets $\{ \bigcup_{i \neq p} W_i(\c) \cap W_p(\c^*) | \c \in C^k \}$ have finite VC-dimension, say $D$. Using the well-known Vapnik-Chervonenkis bound (see, e.g., \cite{Boucheron05}), combined with a bounded difference concentration inequality yields 
         \[
         P_n Y_\c \leq P Y_\c + K \sqrt{\frac{D}{n}}\left ( 1 + \sqrt{\frac{x}{D}} \right ),
         \]
         with probability larger than $1-e^{-x}$, for every $\c$ in $C^k$, and for some absolute constant $K$. Choosing $x = f(n)$ provides the second inequality of Proposition \ref{deviations}.
         
         \subsection{Proof of Lemma \ref{Gaussiancalculus}}\label{ProofofLemmaGaussiancalculus}
         The proof of Proposition \ref{Gaussiancalculus} follows from the Proof of Proposition 3.2 in \cite{Levrard14}. To give an upper bound on $R(\m)$, we may write
         \begin{align*}
         R(\m) & = \sum_{i=1}^{k} \frac{\theta_i}{(2 \pi)^{d/2} N_i \sqrt{| \Sigma_i |}} \sum_{j=1}^{k} \int_{W_j(\m)}{\| x - m_j \|^2 e^{-\frac{1}{2}(x-m_i)^t\Sigma_i^{-1}(x-m_i)} \mathbbm{1}_{\mathcal{B}(0,M)}(x) dx} \\
         & \leq \sum_{i=1}^{k} \frac{\theta_i}{(1-\eta)(2 \pi)^{d/2} \sqrt{| \Sigma_i |}} \int_{\mathbb{R}^d}{\|x - m_i\|^2e^{-\frac{1}{2}(x-m_i)^t\Sigma_i^{-1}(x-m_i)} dx} \\
         & \leq \sum_{i=1}^{k} \frac{\theta_i}{(1-\eta)(2 \pi)^{d/2}} \int_{\mathbb{R}^d}{\| \sqrt{\Sigma_i} u \|^2 e^{-\frac{1}{2}\|u\|^2}du},
         \end{align*}
         where $\sqrt{\Sigma_i}$ denotes the square root of the matrix $\Sigma_i$. Since $\Sigma_i$ has its largest eigenvalue bounded by $\sigma^2$, it follows that $ \| \sqrt{\Sigma_i} u \|^2 \leq \sigma^2 \|u\|^2$, for every $u$ in $\mathbb{R}^d$. We deduce that
         \begin{align*}
         R(\m) & \leq \frac{\sigma^2 k \theta_{max}}{(2\pi)^{d/2}(1-\eta)} \int_{\mathbb{R}^d} {\|u\|^2e^{-\frac{1}{2}\|u\|^2}du} \\
               & \leq \frac{\sigma^2 k \theta_{max} d}{(1-\eta)},
         \end{align*}
         which proves \eqref{meansrisk}. Now let $\c$ be a codebook, and let $i$ be such that $\| c_i - m_j \| > \tau \tilde{B}$,  for every $j$ in $\{1, \hdots, k\}$, with $\tau<1/2$. Since $\mathcal{B}(m_i,\tau \tilde{B} / 2) \subset \mathcal{B}(0,M)$, we may write
         \begin{align*}
         R(\c) & > \int_{\mathcal{B}(m_i,\tau \tilde{B} / 2)}{\left ( \frac{\tau \tilde{B}}{2} \right )^2 \frac{\theta_i}{(2 \pi)^{d/2}\sqrt{|\Sigma_i|}}e^{-\frac{1}{2}(x-m_i)^t\Sigma_i^{-1}(x-m_i)} dx} \\
         & > \frac{\theta_{min}\tau^2 \tilde{B}^2}{4 (2 \pi)^{d/2}}\int_{\sqrt{\Sigma_i}^{-1}\mathcal{B}(0,\tau \tilde{B} / 2)}{e^{-\frac{1}{2}\|u\|^2}du} \\
         & > \frac{\theta_{min}\tau^2 \tilde{B}^2}{4 (2 \pi)^{d/2}}\int_{\mathcal{B}(0,\frac{\tau \tilde{B}}{2\sigma})}{e^{-\frac{1}{2}\|u\|^2}du}.
         \end{align*}
         Since, for every positive $r$, $ \left [ -r/\sqrt{d} , r/\sqrt{d} \right ]^{d} \subset \mathcal{B}(0,r) $, and  $\int_{r}^{\infty}{e^{-r^2/2}dr} \leq e^{-r^2/2}/r$, it follows that
         \begin{align*}
         R(\c) & > \frac{\theta_{min}\tau^2 \tilde{B}^2}{4 (2 \pi)^{d/2}} \left (2 \int_{0}^{\frac{\tau \tilde{B}}{2 \sigma \sqrt{d}}}{e^{-\frac{r^2}{2}}dr} \right )^d\\
               & > \frac{\tau^2 \tilde{B}^2 \theta_{min}}{4} \left ( 1 - \frac{2 \sigma \sqrt{d}}{\sqrt{2\pi } \tau \tilde{B}}e^{-\frac{\tau^2 \tilde{B}^2}{4 d \sigma^2}} \right )^d,                
         \end{align*}
         which proves \eqref{risklowerbound}. At last, let $\tau'$ be such that $2 \tau + \tau' < \frac{1}{2}$, and let $y$ be in $N_{\c^*}(\tau' \tilde{B})$. Then, for every $i$ in $\{1, \hdots, k\}$, we have $\|m_i - y \| \geq \left [\frac{1}{2} - (2 \tau + \tau') \right ] \tilde{B}$. Hence
        \begin{align*}
        f(y) \leq \frac{k \theta_{max}}{(2\pi)^{d/2}(1-\eta) \sigma_-^d} e^{-\frac{\left [\frac{1}{2} - (2 \tau + \tau') \right ]^2 \tilde{B}^2}{2 \sigma^2}}.
        \end{align*}
        Since $\sigma_- \geq c_- \sigma$ and $\lambda(N_{\c^*}(t)) \leq 2 t k M^{d-1}S_{d-1}$, straightforward calculation leads to \eqref{gaussianweightfunction}. 
\bibliography{biblio}
\end{document}